\documentclass[a4paper,12pt,intlimits,oneside,reqno]{amsart}
\usepackage{enumerate}
\usepackage{amsfonts,amsmath, amsxtra,amssymb,latexsym, amscd,amsthm}
\allowdisplaybreaks
\usepackage{latexsym,amssymb}
\newtheorem{theorem}{Theorem}[section]

\theoremstyle{corollary}
\newtheorem{corollary}[theorem]{Corollary}

\theoremstyle{definition}
\newtheorem{definition}[theorem]{Definition}

\theoremstyle{proposition}
\newtheorem{proposition}[theorem]{Proposition}

\theoremstyle{remark}

\numberwithin{equation}{section}

\newcommand{\comment}[1]{}
\allowdisplaybreaks

\begin{document}

\title [Multilinear Hausdorff operator]{Two Weighted estimates for  multilinear Hausdorff Operators on the Morrey-Herz Spaces}
\author{Nguyen Minh Chuong}

\address{Institute of mathematics, Vietnamese  Academy of Science and Technology,  Hanoi, Vietnam.}
\email{nmchuong@math.ac.vn}
\thanks{The first author of this paper is supported by Vietnam National Foundation for Science and Technology Development (NAFOSTED) under grant number 101.02-2014.51}

\author{Dao Van Duong}
\address{School of Mathematics,  Mientrung University of Civil Engineering, Phu Yen, Vietnam.}
\email{daovanduong@muce.edu.vn}

\author{Nguyen Duc Duyet}
\address{Hanoi Pedagogical University 2, Vinh Phuc, Vietnam.}
\email{duyetnguyenduc@gmail.com}
\begin{abstract}
The purpose of this paper is to establish some neccessary and sufficient conditions for the boundedness of a general class of multilinear Hausdorff operators that acts on the product of some two weighted function spaces such as the two weighted Morrey, Herz and Morrey-Herz spaces. Moreover, some sufficient conditions for the boundedness of multilinear Hausdorff operators on the such spaces with respect to the Muckenhoupt weights are also given.
\end{abstract}

\maketitle

\section{Introduction}\label{section1}
The one dimensional Hausdorff operator is defined by
\begin{align}
\mathcal{H}_{\Phi}(f)(x)=\int\limits_{0}^\infty{\frac{\Phi(t)}{t}f\left(\frac{x}{t}\right)dt},
\end{align}
where $\Phi$ is an integrable function on the positive half-line. It is well known that the Hausdorff operator is deeply rooted in the study of one dimensional
Fourier analysis, especially, it is closely related to the summability of the classical Fourier series  (see, for instance,  \cite{CFL2012}, \cite{Liflyand2}, \cite{Hausdorff}, \cite{Hurwitz},  and the references therein). 
\vskip 5pt
It is obvious that if $\Phi(t)=\frac{\chi_{(0,\infty)}(t)}{t}$, then $\mathcal{H}_{\Phi}$ reduces to the Hardy operator defined by
$$\mathcal{H}(f)(x)=\frac{1}{x}\int_0^xf(t)dt,$$
which is one of the most important averaging operators in harmonic analysis. A celebrated Hardy integral inequality \cite{Hardy} can be formulated as follows
$$\|\mathcal{H}(f)\|_{L^p(\mathbb R)}\leq \frac{p}{p-1}\|f\|_{L^p(\mathbb R)}, \text{ \;for all \;} 1<p<\infty,$$
and the constant $\frac{p}{p-1}$ is the best possible. 
\vskip 5pt
It is worth pointing out that if the kernel function $\Phi$ is taken appropriately, then  the Hausdorff operator also reduces to many other classcial operators in analysis such as  the Ces\`{a}ro operator, Hardy-Littlewood-P\'{o}lya operator,  Riemann-Liouville fractional integral operator and Hardy-Littlewood average operator (see, e.g., \cite{An1996, Christ, FGLY2015, Miyachi} and references therein). For further readings on the Hardy type inequality for many classical operators in analysis, see \cite{AS1988, An1996, Batuev-Stepanov, Christ, Heing, Krepela, Kalybay, Prokhorov2003, Prokhorov2013, Prokhorov2013-1, Stempak, Stepanov90} for more details.
\vskip 5pt
The Hausdorff operator is extended to the high dimensional space by Brown and M\'{o}ricz \cite{BM2002} and independently by Lerner and Liflyand \cite{LL2007}. To be more precise, let $\Phi$ be a locally integrable function on $\mathbb R^n$. The Hausdorff operator $\mathcal{H}_{\Phi,A}$ associated to the kernel function $\Phi$ is then defined as
\begin{align}
{\mathcal{H}}_{\Phi,A}(f)(x)=\int_{\mathbb{R}^n}\frac{\Phi(y)}{|y|^n}f(A(y)x)dy, ~ x\in \mathbb{R}^n,
\end{align}
where $\Phi$ is a locally integrable function on $\mathbb{R}^n$ and $A(y)$ is an $n\times n$ matrix satisfying $\det A(y)\ne 0$ for almost everywhere $y$ in the support of $\Phi$. It should be pointed out that if we take  $\Phi(t)=|t|^n\psi(t_1)\chi_{[0,1]^n}(t)$ and $A(t)= t_1.I_n$ ($I_n$ is an identity matrix), for $t=(t_1,t_2,...,t_n)$, where $\psi:[0,1]\to [0,\infty)$ is a measurable function,  $\mathcal{H}_{\Phi,A}$ then reduces to the weighted Hardy-Littlewood average operator due to Carton-Lebrun and Fosset \cite{Carton-Lebrun1984} defined by
\begin{align}
{\mathcal{H}}_\psi f(x)=\int_0^1 f(tx)\psi(t)dt, ~ x\in\mathbb{R}^n.
\end{align} 
Under certain conditions on $\psi$, Carton-Lebrun and Fosset \cite{Carton-Lebrun1984} proved that ${\mathcal{H}}_\psi$ maps $L^p (\mathbb{R}^n)$ into itself for all $1<p<\infty$. They also pointed out that the operator ${\mathcal{H}}_\psi$ commutes with the Hilbert transform when $n=1$, and with certain Calder\'{o}n-Zygmund singular integral operators including the Riesz transform when $n\ge 2$. A further extension of the results obtained in \cite{Carton-Lebrun1984} to the Hardy space and BMO space is due to Xiao \cite{Xiao2001}. 
\vskip 5pt
By letting $\Phi(y)=|y|^n\psi(y_1)\chi_{[0,1]^n}(y)$ and $A(y)=s(y_1).I_n$, where $s:[0,1]\to \mathbb{R}$ is a measurable function, Chuong and Hung \cite{CH2014} studied the operator $U_{\psi,s}$, so-called the generalized Hardy-Ces\`{a}ro operator associated with the parameter curve $s(x,t):=s(t)x$, defined as follows
\begin{align*}
U_{\psi,s} f(x)=\int_0^1 f(s(t)x)\psi(t)dt,\;\; x\in\mathbb R^n.
\end{align*} 
\vskip 5pt
In recent years, the theory of weighted Hardy-Littlewood average operators,
Hardy-Ces\`{a}ro  operators and Hausdorff operators have been significantly developed
into different contexts, and studied on many function spaces such as  Lebesgue, Morrey, Herz, Morrey-Herz, Hardy and BMO spaces including the weighted settings.  For more details, one may find in \cite{Chuong2016, Chuongduong2016, CDH2016, CFL2012, CH2014, Ge, Liflyand2, Li2008, LM2000, LL2007, Miyachi, Mo2005, RF2016} and references therein. On the other hand, it is useful to observe that Coifman and Meyer \cite{Coifman-Meyer75} discovered a multilinear point of view in their study of certain singular integral operators. Thus, the research of the theory of multilinear operators is not only attracted by a pure question to generalize the theory of linear ones but also by their deep applications in harmonic analysis. In 2015, Fu et al. \cite{FGLY2015} introduced the weighted multilinear Hardy operator of the form
\begin{align}
{\mathcal{H}}_\psi^m(\vec{f})(x)=\int_{[0,1]^n} \left(\prod_{i=1}^m f_i(y_ix) \right) \psi(y)dy, ~ x\in\mathbb{R}^n,
\end{align}
where  $\psi: [0,1]^n\to [0,\infty)$ is an integrable function, $\vec{f}=(f_1,...,f_m)$,  $f_i, i=1,...,m,$ are complex-valued measurable functions on $\mathbb{R}^n$.  They  obtained the sharp bounds for ${\mathcal{H}}_\psi^m$ on the product of Lebesgue spaces and central Morrey spaces. As a consequence, these results are also applied to prove sharp estimates of some inequalities due to Riemann-Liouville and Weyl. Later, Hung and Ky \cite{HK2015} studied the weighted multilinear Hardy-Ces\`{a}ro type operators, which are generalized of weighted multilinear Hardy operators, as follows
\begin{align}
U_{\psi,\vec{s}}^{m,n}(\vec{f})(x)=\int_{[0,1]^n} \left(\prod_{i=1}^m f_i(s_i(y)x) \right) \psi(y) dy, ~ x\in\mathbb{R}^n,
\end{align}
where $s_1,...,s_m:[0,1]^n\to \mathbb{R}$. They are also obtained the sharp bounds of weighted multilinear Hardy-Ces\`{a}ro operators that acts on the product of weighted Lebesgue spaces and central Morrey spaces.
\vskip 5pt
Very recently,  Chuong, Duong and Dung \cite{CDD2017} have introduced and studied a more general class of multilinear Hausdorff operators which is defined by
\begin{align}
{\mathcal{H}}_{\Phi,\vec{A}}(\vec{f})(x)=\int_{\mathbb{R}^n}\frac{\Phi(y)}{|y|^n}\prod_{i=1}^m f_i(A_i(y)x)dy,~ x\in \mathbb{R}^n.
\end{align}
Let us take measurable functions $s_1(y), ..., s_m(y)\not=0$  almost everywhere in $\mathbb R^n$. Consider a special case when the matrices $A_i(y)={\rm diag }[ s_{i}(y),...,s_{i}(y) ]$, for all $i=1,...,m$. Then, we aslo study in this paper the hybrid multilinear operator of the form as follows
\begin{align}\label{var-multi-haus}
\mathcal{H}_{\phi, \vec{s}}(\vec{f})(x)=\int_{\mathbb{R}^n}\left(\prod_{i=1}^m f_i(s_i(y)x)\right)\phi(y)dy,~ x\in \mathbb{R}^n.
\end{align}
By letting $\phi(y)=\psi(y)\chi_{[0,1]^n}(y)$, it is clear that the operator $\mathcal{H}_{\phi, \vec{s}}$ reduces to the operator $U_{\psi,\vec{s}}^{m,n}$. The multilinear Hausdorff operators are extended to study on some function spaces in the real field as well as $p$-adic numbers field. The interested reader is refered to the works \cite{CDD2017, HausdoffCDD1, CDD2019-1, CHH2017, FGLY2015, HK2015, Krepela} for more details.
\vskip 5pt
It is very important to study weighted estimates for many classcial operators in
harmonic analysis. B. Muckenhoupt [30] first discovered the weighted norm
inequality for the Hardy-Littlewood maximal operators in the real setting and, then, the class of $A_p$ weights is introduced, which is well known in harmonic analysis. Moreover, Coifman and Fefferman \cite{CF1974} extended the theory of Muckenhoupt
weights to general Calder\'{o}n-Zygmund operators. They also proved that $A_p$
weights satisfy the crucial reverse H\"{o}lder condition. For further information on the Muckenhoupt weights as well as its deep applications in harmonic analysis, in the theory of extrapolation of operators and in nonlinear partial differential equations,  see \cite{CEKMM2007, GrafakosModern, Stein1993} and the references therein. In the recent years, there is an increasing interest on the study of the problems concerning the two-weight norm inequality for many fundamental operators in harmonic anlysis, for example, such as maximal operator, the Hilbert transform, singular integral operator, Hardy operator and Hausdorff operator. More details, one may find in \cite{CDD2019-2,  Krepela, Sawyer, Stepanov2011, Tanaka} and the references therein. It is therefore of interest to extend the study of the two-weight norm inequalities for the multilinear Hausdorff operators. In this paper, we give some neccessary and sufficient conditions for the boundedness of the multilinear Hausdorff operators that acts on the product of  the two weighted Morrey, Herz and Morrey-Herz spaces. Also, in the setting of the function spaces with the Muckenhoupt weights, some sufficient conditions for the boundedness of multilinear Hausdorff operators are also discussed.
\vskip 5pt
Our paper is organized as follows. In Section 2, we introduce the some weighted function spaces such as weighted Lebesgue spaces, two weighted central Morrey spaces, two weighted Herz spaces and two weighted Morrey-Herz spaces. Besides that, we also introduce the class of Muckenhoupt weights. The main theorems and their proofs in this paper are given in Section 3.
\section{Preliminaries}\label{section2}
We start with this section by recalling some standard definitions and notations. By $\|T\|_{X\to Y}$, we denote the norm of $T$ between two normed vector spaces $X$ and $Y$. The letter $C$ denotes a positive constant which is independent of the main parameters, but may be different from line to line. For any $a\in\mathbb R^n$ and $r>0$, we shall denote by $B(a,r)$ the ball centered at $a$ with radius $r$. We also denote $S_{n-1}=\{ x\in\mathbb{R}^n: |x|=1\}$ and $|S_{n-1}|=\frac{2\pi^{\frac{n}{2}}}{\Gamma\left( \frac{n}{2}\right)}$. For any real number $p>0$, denote by $p'$ conjugate real number of $p$, i.e. $\frac{1}{p}+\frac{1}{p'}=1$.
\vskip 5pt
Next, we write $a \lesssim b$ to mean that there is a positive constant $C$ , independent of the main parameters, such that $a\leq Cb$. The symbol $f\simeq g$ means that $f$ is equivalent to $g$ (i.e.~$C^{-1}f\leq g\leq Cf$). Throughout the paper,  the weighted functions will be denoted local integral nonnegative  measurable functions on $\mathbb{R}^n$. For any measurable set $E$, we denote by $\chi_E$ its characteristic function, by $|E|$ its Lebesgue measure, and $\omega(E)=\int_{E}\omega(x)dx$ for any weighted function $\omega$. Let $L^q_{\omega}(\mathbb{R}^n)$ $(0<q<\infty)$ be the space of all Lebesgue measurable functions $f$ on $\mathbb{R}^n$ such that
\begin{align*}
\|f\|_{L^q_{\omega}(\mathbb{R}^n)}=\left(\int_{\mathbb{R}^n}|f(x)|^q\omega(x)dx \right)^{\frac{1}{q}}<\infty.
\end{align*} 
The space $L^q_\text {loc}(\omega, \mathbb R^n)$ is defined as the set of all measurable functions $f$ on $\mathbb R^n$ satisfying $\int_{K}|f(x)|^q\omega(x)dx<\infty$ for any compact subset $K$ of $\mathbb R^n$. The space $L^q_\text {loc}(\omega, \mathbb R^n\setminus\{0\})$ is also defined in a similar way to the space  $L^q_\text {loc}(\omega, \mathbb R^n)$.
\vskip 5pt
In what follows, denote $\chi_k=\chi_{C_k}$, $C_k=B_k\setminus B_{k-1}$ and $B_k = \big\{x\in \mathbb R^n: |x| \leq 2^k\big\}$, for all $k\in\mathbb Z.$  
Now, we are in a position to give some definitions of the two weighted Morrey, Herz  and Morrey-Herz spaces. For further information on these spaces as well as their deep applications in analysis, the interested readers may refer to the works  \cite{ALP2000}, \cite{CDH2016}, and \cite{CHH2017}, especially, to the monograph \cite{LYH2008} and references therein.
\begin{definition}
Let $0<q<\infty$ and $1+\lambda q>0, \lambda\in\mathbb{R}$. Suppose $v, \omega$ are  two weighted functions. Then,  the two weighted central Morrey space is defined by
\begin{align*}
\dot{M}^{q,\lambda}_{v,\omega}(\mathbb{R}^n)=\{f\in L^q_{\rm loc}(\omega): \|f\|_{\dot{M}^{q,\lambda}_{v,\omega}(\mathbb{R}^n)}<\infty \},
\end{align*}
where
\begin{align*}
\|f\|_{\dot{M}^{q,\lambda}_{v,\omega}(\mathbb{R}^n)}=\sup\limits_{R>0}\frac{1}{v(B(0,R))^{\lambda +\frac{1}{q}}}\|f\|_{L^q_\omega (B(0,R))}.
\end{align*}
\end{definition}
In particular, if $v=\omega$, then we write $\dot{M}^{q,\lambda}_{v,\omega}(\mathbb{R}^n)=\dot{M}^{q,\lambda}_{\omega}(\mathbb{R}^n)$ which is called the weighted central Morrey space.

\begin{definition}
Let $0<p<\infty, 1<q<\infty$ and $\alpha\in\mathbb{R}$. Let $v$ and $\omega$ be two  weighted functions. The homogeneous two weighted Herz space $\dot{K}^{\alpha,p,q}_{v,\omega}(\mathbb{R}^n)$ is defined to be the set of all $f\in L_{\rm loc}^q(\mathbb{R}^n\backslash\{0\},\omega)$ such that
\begin{align*}
\|f\|_{\dot{K}^{\alpha,p,q}_{v,\omega}(\mathbb{R}^n)}=\left(\sum\limits_{k=-\infty}^\infty  v(B_k)^{\frac{\alpha p}{n}}\|f\chi_k\|^p_{L^q_\omega (\mathbb{R}^n)} \right)^{\frac{1}{p}}<\infty.
\end{align*}
\end{definition}
Remark that if $v$ is a constant function, then $\dot{K}^{\alpha,p,q}_{v,\omega}(\mathbb{R}^n):=\dot{K}^{\alpha}_{p,q}(\omega, \mathbb{R}^n)$ is the weighted Herz space studied in \cite{CDH2016} and \cite{LYH2008}.
\begin{definition}
Let $\alpha\in\mathbb{R}, 0<p< \infty,\lambda\ge 0$ and $v, \omega$ be weighted functions. Then, the two weighted Morrey-Herz space $M\dot{K}_{v,\omega}^{\alpha,\lambda,p,q}(\mathbb{R}^n)$ is defined as the space of all functions $f\in L_{\rm loc}^q(\mathbb{R}^n\backslash\{0\},\omega)$ such that $\|f\|_{M\dot{K}_{v,\omega}^{\alpha,\lambda,p,q}(\mathbb{R}^n)}<\infty$, where

\begin{align*}
\|f\|_{M\dot{K}_{v,\omega}^{\alpha,\lambda,p,q}(\mathbb{R}^n)}=\sup\limits_{k_0\in\mathbb{Z}}\left(v(B_{k_0})^{-\frac{\lambda}{n}}\left(\sum\limits_{k=-\infty}^{k_0}v(B_k)^{\frac{\alpha p}{n} }\|f\chi_k\|^p_{L^q_\omega (\mathbb{R}^n)} \right)^{\frac{1}{p}}\right).
\end{align*}
\end{definition}
In particular, when $\omega=v$, we denote $M\dot{K}_{p,q}^{\alpha,\lambda}(\omega,\mathbb{R}^n)$ instead of $M\dot{K}_{v,\omega}^{\alpha,\lambda,p,q}(\mathbb{R}^n)$.
Note that if $\lambda=0$, it is easy to see that $M\dot{K}_{v,\omega}^{\alpha,\lambda,p,q}(\mathbb{R}^n)=\dot{K}_{v,\omega}^{\alpha,p,q}(\mathbb{R}^n)$. Consequently, the two weighted Herz space is a special case of the two weighted Morrey-Herz space.  Some applications of two weighted Morrey-Herz spaces to the Hardy-Ces\`{a}ro operators may be found, for example, in the works \cite{CDH2016}, \cite{CHH2017}. It should be pointed out that the Herz spaces are natural generalization of the Lebesgue spaces with power weight.
\vskip 5pt
Next, we will recall some preliminaries on the theory of $A_p$ weights which was first introduced by Benjamin Muckenhoupt \cite{Mu1972} in the Euclidean spaces in order to study the weighted $L^p$ boundedness of Hardy-Littlewood maximal functions. Let us recall that a weight is a nonnegative, locally integrable function on $\mathbb{R}^n$.

\begin{definition}
Let $1<\xi<\infty$. It is said that a weight $\omega\in A_\xi(\mathbb{R}^n)$ if there exists a constant $C>0$ such that for all balls $B\subset\mathbb R^n$,
\begin{align*}
\left(\frac{1}{|B|}\int_B \omega(x)dx \right) \left(\frac{1}{|B|}\int_B \omega(x)^{-\frac{1}{\xi-1}}dx \right)^{\xi-1}\le C.
\end{align*}
We say that a weight $\omega\in A_1(\mathbb{R}^n)$ if there is a constant $C>0$ such that for all balls $B\subset\mathbb R^n$,
\begin{align*}
\frac{1}{|B|}\int_B\omega(x)dx\le C \mathop {\rm essinf}\limits_{x\in B} \omega(x).
\end{align*}
We denote
\begin{align*}
A_\infty(\mathbb{R}^n)=\mathop{\cup}\limits_{1\le \xi <\infty}A_\xi(\mathbb{R}^n).
\end{align*}
\end{definition}
Let us now recall the following standard results related to the Muckenhoupt weights. For further readings on the class of the Muckenhoupt weights as well as its deep applications, see in the monographs \cite{GrafakosModern} and \cite{Stein1993}.

\begin{proposition}
(i) $A_\xi(\mathbb{R}^n)\varsubsetneq A_\eta(\mathbb{R}^n)$, for $1\le\xi<\eta<\infty$.\\
(ii) If $\omega \in A_\xi(\mathbb{R}^n), 1<\xi<\infty$, then there is an $\varepsilon>0$ such that $\xi-\varepsilon>1$ and $\omega\in A_{\xi-\varepsilon}(\mathbb{R}^n)$.
\end{proposition}

A close relation to $A_\infty(\mathbb{R}^n)$ is the crucial reverse H\"{o}lder condition due to Coifman and Fefferman \cite{CF1974}. If there exist $r>1$ and a fixed constant $C$ such that
\begin{align*}
\left(\frac{1}{|B|}\int_B\omega(x)^r dx  \right)^{\frac{1}{r}}\le\frac{C}{|B|}\int_B\omega(x)dx,
\end{align*}
for all balls $B\subset \mathbb{R}^n$, we then say that $\omega$ satisfies the reverse H\"{o}lder condition of order $r$ and write $\omega\in RH_r$. According to Theorem 19 and Corollary 21 in \cite{IMS2015}, $\omega\in A_\infty$ if and only if there exists some $r>1$ such that $\omega\in RH_r$. Moreover, if $\omega\in RH_r,r>1$ then $\omega\in RH_{r+\varepsilon}$ for some $\varepsilon>0$. We thus write $r_\omega=\sup \{r>1:\omega\in RH_r \}$ to denote the critical index of $\omega$ for the reverse H\"{o}lder condition.

\begin{proposition}\label{prop2.1}
If $\omega\in A_\xi(\mathbb{R}^n),1\le  \xi<\infty,$ then for any $f\in L^1_{\text{loc}}(\mathbb{R}^n)$ and any ball $B\subset \mathbb{R}^n$,
\begin{align*}
\frac{1}{|B|}\int_B|f(x)|dx\le C\left(\frac{1}{\omega(B)}\int_B|f(x)|^{\xi}\omega(x)dx \right)^{\frac{1}{\xi}}.
\end{align*}
\end{proposition}

\begin{proposition}\label{prop2.2}
Let $\omega\in A_\xi\cap RH_\delta$, for $\xi\ge 1$ and $\delta>1$. Then, there exist two constants $C_1, C_2>0$ such that
\begin{align*}
C_1\left(\frac{|E|}{|B|} \right)^\xi\le \frac{\omega(E)}{\omega(B)}\le C_2\left(\frac{|E|}{|B|} \right)^{\frac{\delta-1}{\delta}},
\end{align*}
for any measurable subset $E$ of a ball $B$. In particular, for any $\lambda>1$, we have
\begin{align*}
\omega(B(x_0,\lambda R))\le C\lambda^{n\xi}\omega(B(x_0,R)).
\end{align*}
\end{proposition}

\section{Main results and their proofs}
Before stating our main results, we introduce some notations that are often used in this section. Throughout this section, $\beta, \gamma,$ $ \beta_1, \gamma_1, ...,\beta_m, \gamma_m$ are real numbers greater than $-n$, and $1\leq p, q<\infty$, $1\leq p_i,q_i<\infty$ for all $i=1,...,m$ satisfying
\begin{align*}
\frac{1}{p_1}+\cdots+\frac{1}{p_m}&=\frac{1}{p},\\
\frac{1}{q_1}+\cdots+\frac{1}{q_m}&=\frac{1}{q}.\\
\end{align*}
For a matrix $A=(a_{ij})_{n\times n}$, we define the norm of $A$ as follows
\begin{align*}
\|A\|=\left(\sum\limits_{i,j=1}^n |a_{ij}|^2 \right)^{\frac{1}{2}}.
\end{align*}
It is easy to see that $|Ax|\le \|A\||x|$ for any vector $x\in\mathbb{R}^n$. In particular, if $A$ is invertible,  we then have
\begin{align*}
\|A\|^{-n}\le |\det (A^{-1})|\le \|A^{-1}\|^n.
\end{align*}
In the first part of this section, we will discuss the boundedness of multilinear Hausdorff operators on the two weighted Morrey, Herz, and Morrey-Herz spaces with power weights provided that
\begin{align}\label{eq3.1}
\rho_{\vec{A}}:=\mathop {\text{ess}\sup }\limits_{t\in\mathbb R^n, i=1,...,m}\|A_i(t)\|.\|A^{-1}_i(t)\|<\infty.
\end{align}

As some applications, we also obtain the boundedness and bounds of the multilinear Hardy-Ces\`{a}ro operators on the such spaces. 

Observe that if $A(t)$  is a real orthogonal matrix for almost everywhere $t$ in $\mathbb R^n$, then $A(t)$ satisfies the condtion \eqref{eq3.1}.
It is useful to remark that the condition \eqref{eq3.1} implies $\|A_i(t)\|\simeq\|A^{-1}_i(t)\|^{-1}$. Moreover, it is easily seen that  
\begin{align}\label{eq3.2}
\|A_i(t)\|^\eta\lesssim \|A^{-1}_i(t)\|^{-\eta},~\text{for all} ~ \eta\in\mathbb{R}, 
\end{align}
and
\begin{align}\label{eq3.3}
|A_i(t)x|^\eta\gtrsim \|A^{-1}_i(t)\|^{-\eta}.|x|^\eta, ~\text{for all} ~ \eta\in\mathbb{R},\; x\in\mathbb{R}^n.
\end{align}
\vskip 5pt
Now, we are in a position to give our first main results concerning the
boundedness of multilinear Hausdorff operators on two weighted central Morrey spaces.
\begin{theorem}\label{theo3.1}
Let $v(x)=|x|^\beta, \omega(x)=|x|^\gamma$, $v_i(x)=|x|^{\beta_i}, \omega_i(x)=|x|^{\gamma_i}, \textit{ for all } i=1,...,m$, and the following conditions hold $$\sum\limits_{i=1}^m \frac{\beta_i}{q_i}=\frac{\beta}{q},\; \sum\limits_{i=1}^m\left(\frac{n+\beta_i}{n+\beta}\right)\lambda_i=\lambda, \textit{ and } \sum\limits_{i=1}^m \frac{\gamma_i}{q_i}=\frac{\gamma}{q}.$$ 
{\rm(i)}  If 
\begin{align*}
\mathcal{C}_{1}=\int_{\mathbb{R}^n}\frac{|\Phi(y)|}{|y|^n}\prod_{i=1}^m\|A^{-1}_i(y)\|^{-(\beta_i+n)\lambda_i+(\gamma_i-\beta_i)\frac{1}{q_i}} dy<\infty,
\end{align*}
then ${\mathcal{H}}_{\Phi,\vec{A}}$ is bounded from $\prod_{i=1}^m\dot{M}^{q_i,\lambda_i}_{v_i, \omega_i}(\mathbb{R}^n)$ to $\dot{M}^{q,\lambda}_{v,\omega}(\mathbb{R}^n)$. Moreover,
\begin{align*}
\|{\mathcal{H}}_{\Phi,\vec{A}} \|_{\prod_{i=1}^m\dot{M}^{q_i,\lambda_i}_{v_i, \omega_i}(\mathbb{R}^n)\to \dot{M}^{q,\lambda}_{v,\omega}(\mathbb{R}^n)}\lesssim \mathcal{C}_{1}.
\end{align*}
{\rm(ii)} Conversely, suppose $\Phi$ is a real function with a constant sign in $\mathbb{R}^n$. Then, if ${\mathcal{H}}_{\Phi,\vec{A}}$ is bounded from $\prod_{i=1}^m\dot{M}^{q_i,\lambda_i}_{v_i, \omega_i}(\mathbb{R}^n)$ to $\dot{M}^{q,\lambda}_{v,\omega}(\mathbb{R}^n)$, we have $\mathcal{C}_1<\infty$. Furthermore,
\begin{align*}
\|{\mathcal{H}}_{\Phi,\vec{A}} \|_{\prod_{i=1}^m\dot{M}^{q_i,\lambda_i}_{v_i, \omega_i}(\mathbb{R}^n)\to \dot{M}^{q,\lambda}_{v,\omega}(\mathbb{R}^n)}\gtrsim \mathcal{C}_{1}.
\end{align*}
\end{theorem}

\begin{proof}
(i) For $\sum\limits_{i=1}^m \frac{\gamma_i}{q_i}=\frac{\gamma}{q}$, it is easy to see that $\prod_{i=1}^m \omega_i^{\frac{1}{q_i}} (x)=\omega^{\frac{1}{q}}(x)$. Then,
 by the Minkowski inequality, we have
\begin{align*}
\|{\mathcal{H}}_{\Phi,\vec{A}}(\vec{f})\|_{L^q_\omega (B(0,R))}
&=\left(\int_{B(0,R)}\left|\int_{\mathbb{R}^n}\frac{\Phi(y)}{|y|^n}\prod_{i=1}^mf_i(A_i(y).x)dy\right|^q\omega(x)dx \right)^{\frac{1}{q}}\\
&=\left(\int_{B(0,R)}\left|\int_{\mathbb{R}^n}\frac{\Phi(y)}{|y|^n}\prod_{i=1}^mf_i(A_i(y).x)\omega_i^{\frac{1}{q_i}}(x)dy\right|^qdx \right)^{\frac{1}{q}}\\
&\le \int_{\mathbb{R}^n}\frac{|\Phi(y)|}{|y|^n}\left(\Big\Vert\prod_{i=1}^mf_i(A_i(y).x)\omega_i^{\frac{1}{q_i}}(x) \Big\Vert_{L^q (B(0,R))} \right)dy.
\end{align*}
From assuming that $\sum\limits_{i=1}^m \frac{1}{q_i}=\frac{1}{q}$ and applying the H\"{o}lder inequality, it immediately follows that
\begin{align*}
\Big\Vert\prod_{i=1}^mf_i(A_i(y).x)\omega_i^{\frac{1}{q_i}}(x) \Big\Vert_{L^q (B(0,R))}\le \prod_{i=1}^m \Vert f_i(A_i(y).x) \Vert_{L^{q_i}_{\omega_i}( B(0,R))}.
\end{align*}
So, we obtain 
\begin{align*}
\|{\mathcal{H}}_{\Phi,\vec{A}}(\vec{f})\|_{L^q_\omega (B(0,R))} &\mathop  \le\int_{\mathbb{R}^n}\frac{|\Phi(y)|}{|y|^n}\left(\prod_{i=1}^m\Vert f_i(A_i(y).x) \Vert_{L^{q_i}_{\omega_i}( B(0,R))} \right)dy.
\end{align*}
Using change of variable  $z=A_i(y).x$ with $z\in A_i(y).B(0,R)$, we have
\begin{align*}
&\|{\mathcal{H}}_{\Phi,\vec{A}}(\vec{f})\|_{L^q_\omega (B(0,R))}\\
&\mathop  \le  \int_{\mathbb{R}^n}\frac{|\Phi(y)|}{|y|^n}\left(\prod_{i=1}^m\int_{A_i(y).B(0,R)} |f_i(z)|^{q_i}|A^{-1}(y).z|^{\gamma_i} |\det A^{-1}_i(y)|dz \right)^{\frac{1}{q_i}}dy\\
&\le \int_{\mathbb{R}^n}\frac{|\Phi(y)|}{|y|^n}\left(\prod_{i=1}^m\int_{B(0,\|A_i(y)\|.R)} |f_i(z)|^{q_i}|A^{-1}(y).z|^{\gamma_i} |\det A^{-1}_i(y)|dz \right)^{\frac{1}{q_i}}dy.
\end{align*}
By \eqref{eq3.2}, for all $\gamma_i\in\mathbb{R}$, we have the following useful inequality
\begin{align*}
|A_i^{-1}(y).z|^{\gamma_i} \le \max\{\|A_i^{-1}(y)\|^{\gamma_i},\|A_i(y)\|^{-\gamma_i} \}|z|^{\gamma_i}.
\end{align*}
Consequently, we also obtain
\begin{align}\label{basic-ineq}
\|{\mathcal{H}}_{\Phi,\vec{A}}(\vec{f})\|_{L^q_\omega (B(0,R))}
&\le  \int_{\mathbb{R}^n}\frac{|\Phi(y)|}{|y|^n}\left(\prod_{i=1}^m\max\{\|A_i^{-1}(y)\|^{\gamma_i},\|A_i(y)\|^{-\gamma_i} \}|\det A^{-1}_i(y)|\times\right.\notag\\
&\left.\times\int\limits_{B(0,\|A_i(y)\|.R)} |f_i(z)|^{q_i}.\omega_i(z) dz \right)^{\frac{1}{q_i}}dy\notag\\
&\le  \int_{\mathbb{R}^n}\frac{|\Phi(y)|}{|y|^n}\prod_{i=1}^m\max\{\|A_i^{-1}(y)\|^{\gamma_i},\|A_i(y)\|^{-\gamma_i} \}^{\frac{1}{q_i}}|\det A^{-1}_i(y)|^{\frac{1}{q_i}}  \times\notag\\
&\times\|f_i\|_{L^{q_i}_{\omega_i}(B(0,\|A_i(y)\|.R))} dy.
\end{align}
By the definition of two weighted central  Morrey space, we get
\begin{align}\label{eq3.4}
&\|{\mathcal{H}}_{\Phi,\vec{A}}(\vec{f})\|_{\dot{M}^{q,\lambda}_{v,\omega}(\mathbb{R}^n)}\notag\\
&\le \sup\limits_{R>0}\frac{1}{v(B(0,R))^{\lambda +\frac{1}{q}}}  \int_{\mathbb{R}^n}\frac{|\Phi(y)|}{|y|^n}\prod_{i=1}^m\max\{\|A_i^{-1}(y)\|^{\gamma_i},\|A_i(y)\|^{-\gamma_i} \}^{\frac{1}{q_i}}|\det A^{-1}_i(y)|^{\frac{1}{q_i}}   \times\notag\\
&\;\;\;\;\;\times\|f_i\|_{L^{q_i}_{\omega_i}(B(0,\|A_i(y)\|.R))} dy\notag\\
&\le \sup\limits_{R>0}  \int_{\mathbb{R}^n}\frac{|\Phi(y)|}{|y|^n}\prod_{i=1}^m\max\{\|A_i^{-1}(y)\|^{\gamma_i},\|A_i(y)\|^{-\gamma_i} \}^{\frac{1}{q_i}}|\det A^{-1}_i(y)|^{\frac{1}{q_i}}\times\notag\\
&\;\;\;\;\;\times\frac{1}{v(B(0,R))^{\lambda +\frac{1}{q}}}   \|f_i\|_{L^{q_i}_{\omega_i}(B(0,\|A_i(y)\|.R))} dy.
\end{align} 
Remark that by the conditions $\lambda+\frac{1}{q}>0, \lambda_i+\frac{1}{q_i}>0$, $i=1,...,m$, and $\sum\limits_{i=1}^m (\beta_i+n)(\lambda_i +\frac{1}{q_i})=(\beta+n)\left(\lambda+\frac{1}{q}\right)$,
it is easy to see that
$$v(B(0,R))^{\lambda +\frac{1}{q}} \lesssim R^{(\beta+n)(\lambda +\frac{1}{q})},$$
and
\begin{align*}
 \prod_{i=1}^m\frac{ v_i(B(0,\|A_i(y)\|.R))^{(\lambda_i +\frac{1}{q_i})}}{v(B(0,R))^{\lambda +\frac{1}{q}}}&\lesssim \prod_{i=1}^m \frac{(\|A_i(y)\|.R)^{(\beta_i+n)(\lambda_i +\frac{1}{q_i})}}{R^{(\beta+n)(\lambda +\frac{1}{q})}}\\
&=\frac{R^{\sum\limits_{i=1}^m (\beta_i+n)(\lambda_i +\frac{1}{q_i})} }{R^{(\beta+n)(\lambda +\frac{1}{q})}}\prod_{i=1}^m \|A_i(y)\|^{(\beta_i+n)(\lambda_i +\frac{1}{q_i})}\\
&=\prod_{i=1}^m \|A_i(y)\|^{(\beta_i+n)(\lambda_i +\frac{1}{q_i})}.
\end{align*}
This shows that
\begin{align*}
\frac{1}{v(B(0,R))^{\lambda +\frac{1}{q}}}\lesssim \prod_{i=1}^m \frac{\|A_i(y)\|^{(\beta_i+n)(\lambda_i +\frac{1}{q_i})}}{v_i(B(0,\|A_i(y)\|.R))^{(\lambda_i +\frac{1}{q_i})}}.
\end{align*}
By the above estimations, we have
\begin{align*}
&\|{\mathcal{H}}_{\Phi,\vec{A}}(\vec{f})\|_{\dot{M}^{q,\lambda}_{v,\omega}(\mathbb{R}^n)}\\
&\le \sup\limits_{R>0}  \int_{\mathbb{R}^n}\frac{|\Phi(y)|}{|y|^n}\prod_{i=1}^m\max\{\|A_i^{-1}(y)\|^{\gamma_i},\|A_i(y)\|^{-\gamma_i} \}^{\frac{1}{q_i}}|\det A^{-1}_i(y)|^{\frac{1}{q_i}}\times\\
&\;\;\;\;\;\times\prod_{i=1}^m\frac{ \|A_i(y)\|^{(\beta_i+n)(\lambda_i +\frac{1}{q_i})}}{v_i(B(0,\|A_i(y)\|.R))^{(\lambda_i +\frac{1}{q_i})}}  \|f_i\|_{L^{q_i}_{\omega_i}(B(0,\|A_i(y)\|.R))} dy\\
&\le   \int_{\mathbb{R}^n}\frac{|\Phi(y)|}{|y|^n}\prod_{i=1}^m\max\{\|A_i^{-1}(y)\|^{\gamma_i},\|A_i(y)\|^{-\gamma_i} \}^{\frac{1}{q_i}}|\det A^{-1}_i(y)|^{\frac{1}{q_i}} \|A_i(y)\|^{(\beta_i+n)(\lambda_i +\frac{1}{q_i})} dy\times\\
&\;\;\;\;\;\times\prod_{i=1}^m \|f_i\|_{\dot{M}^{q_i,\lambda_i}_{v_i, \omega_i}(\mathbb{R}^n)}.
\end{align*}
By the inequality \eqref{eq3.2} and the property of invertible matrices, we get
$$|\det A_i^{-1}(y)|^{\frac{1}{q_i}} \le \|A^{-1}_i(y)\|^{\frac{n}{q_i}},$$ 
$$\|A_i(y)\|^{(\beta_i+n)(\lambda_i +\frac{1}{q_i})}\lesssim \|A^{-1}_i(y)\|^{-(\beta_i+n)(\lambda_i +\frac{1}{q_i})},$$
 and 
$$\prod_{i=1}^m\max\{\|A_i^{-1}(y)\|^{\gamma_i},\|A_i(y)\|^{-\gamma_i} \}^{\frac{1}{q_i}}\lesssim \prod_{i=1}^m\|A^{-1}_i(y)\|^{\frac{\gamma_i}{q_i}}.$$
This implies that
\begin{align}\label{eq3.5}
&\prod_{i=1}^m\max\{\|A_i^{-1}(y)\|^{\gamma_i},\|A_i(y)\|^{-\gamma_i} \}^{\frac{1}{q_i}}|\det A^{-1}_i(y)|^{\frac{1}{q_i}}\|A_i(y)\|^{(\beta_i+n)(\lambda_i +\frac{1}{q_i})}\notag\\
&\lesssim\prod_{i=1}^m\|A^{-1}_i(y)\|^{-(\beta_i+n)\lambda_i+(\gamma_i-\beta_i)\frac{1}{q_i}}.
\end{align}
Consequently, 
$$\|{\mathcal{H}}_{\Phi,\vec{A}}(\vec{f})\|_{\dot{M}^{q,\lambda}_{v,\omega}(\mathbb{R}^n)}\lesssim
\mathcal{C}_1.\prod_{i=1}^m \|f_i\|_{\dot{M}^{q_i,\lambda_i}_{v_i, \omega_i}(\mathbb{R}^n)}. $$
Therefore, the operator ${\mathcal{H}}_{\Phi,\vec{A}}$ is bounded from the product space $\prod_{i=1}^m\dot{M}^{q_i,\lambda_i}_{v_i, \omega_i}(\mathbb{R}^n)$ to $\dot{M}^{q,\lambda}_{v,\omega}(\mathbb{R}^n)$. The proof for the part (i) of the theorem is finished.
\vskip 5pt
(ii) Conversely, suppose that ${\mathcal{H}}_{\Phi,\vec{A}}$ is bounded from the product $\prod_{i=1}^m\dot{M}^{q_i,\lambda_i}_{v_i, \omega_i}(\mathbb{R}^n)$ to $\dot{M}^{q,\lambda}_{v,\omega}(\mathbb{R}^n)$. Then, let us choose
\begin{align*}
f_i(x)=|x|^{(\beta_i+n)\lambda_i+(\beta_i-\gamma_i)\frac{1}{q_i}}.
\end{align*}
It is evident that $\|f_i\|_{\dot{M}^{q,\lambda}_{v,\omega}(\mathbb{R}^n)}>0$, for all $i=1,...,m$. Now, we need to show that
\begin{align*}
\|f_i\|_{\dot{M}^{q,\lambda}_{v,\omega}(\mathbb{R}^n)}<\infty, ~\text{for all} ~ i=1,...,m.
\end{align*}
Indeed, we have
\begin{align*}
\|f_i\|_{\dot{B}^{q_i,\lambda_i}_{v_i, \omega_i}(\mathbb{R}^n)}
&=\sup\limits_{R>0}\frac{1}{v_i(B(0,R))^{\lambda_i +\frac{1}{q_i}}}\left(\int_{B(0,R)}|x|^{(\beta_i+n)\lambda_iq_i+(\beta_i-\gamma_i)}.|x|^{\gamma_i}dx \right)^{\frac{1}{q_i}}\\
&=\sup\limits_{R>0}\frac{1}{v_i(B(0,R))^{\lambda_i +\frac{1}{q_i}}}\left(\int_{B(0,R)}|x|^{(\beta_i+n)\lambda_iq_i+\beta_i}dx \right)^{\frac{1}{q_i}}\\
&=\sup\limits_{R>0} \frac{1}{R^{(\beta_i+n)\left(\lambda_i+\frac{1}{q_i}\right)}} \left(\int_0^R\int_{S_{n-1}}|rx'|_h^{(\beta_i+n)\lambda_iq_i+\beta_i}.r^{n-1} d\sigma(x')dr  \right)^{\frac{1}{q_i}}\\
&\simeq \sup\limits_{R>0} \frac{1}{R^{(\beta_i+n)\left(\lambda_i+\frac{1}{q_i}\right)}}\frac{R^{(\beta_i+n)\left(\lambda_i+\frac{1}{q_i}\right)}}{(\beta_i+n)(\lambda_i+\frac{1}{q_i})}\simeq\frac{1} {(\beta_i+n)(\lambda_i+\frac{1}{q_i})}<\infty.
\end{align*}
Thus, by choosing $f_i$ and the condition \eqref{eq3.3}, we conclude that
\begin{align*}
&\|{\mathcal{H}}_{\Phi,\vec{A}}(\vec{f})\|_{\dot{M}^{q,\lambda}_{v,\omega}(\mathbb{R}^n)}\\
&=\sup\limits_{R>0}\frac{1}{v(B(0,R))^{\lambda +\frac{1}{q}}}\left(\int_{B(0,R)}\left|\int_{\mathbb{R}^n}\frac{\Phi(y)}{|y|^n}\prod_{i=1}^m|A_i(y)x|^{(\beta_i+n)\lambda_i+(\beta_i-\gamma_i)\frac{1}{q_i}}dy\right|^q\omega(x)dx \right)^{\frac{1}{q}}\\
&\gtrsim \int_{\mathbb{R}^n}\frac{|\Phi(y)|}{|y|^n}\prod_{i=1}^m \|A_i^{-1}(y)\|^{-\left((\beta_i+n)\lambda_i+(\beta_i-\gamma_i)\frac{1}{q_i} \right)}dy\times\\
&\;\;\;\;\;\times \sup\limits_{R>0}\frac{1}{v(B(0,R))^{\lambda +\frac{1}{q}}}\left(\int_{B(0,R)}\left|\prod_{i=1}^m |x|^{(\beta_i+n)\lambda_i+(\beta_i-\gamma_i)\frac{1}{q_i}}\right|^q |x|^\gamma dx \right)^{\frac{1}{q}}\\
&\gtrsim \int_{\mathbb{R}^n}\frac{|\Phi(y)|}{|y|^n}\prod_{i=1}^m \|A_i^{-1}(y)\|^{-\left((\beta_i+n)\lambda_i+(\beta_i-\gamma_i)\frac{1}{q_i} \right)}dy\times\\
&\;\;\;\;\;\times \sup\limits_{R>0}\frac{1}{v(B(0,R))^{\lambda +\frac{1}{q}}}\left(\int_{B(0,R)} |x|^{\sum\limits_{i=1}^m\left((\beta_i+n)\lambda_i+(\beta_i-\gamma_i)\frac{1}{q_i}\right)q+\gamma} dx \right)^{\frac{1}{q}}\\
&\gtrsim \left(\int_{\mathbb{R}^n}\frac{|\Phi(y)|}{|y|^n}\prod_{i=1}^m \|A_i^{-1}(y)\|^{-(\beta_i+n)\lambda_i+(\gamma_i-\beta_i)\frac{1}{q_i} }dy\right)
\times\prod_{i=1}^m \|f_i\|_{\dot{M}^{q_i,\lambda_i}_{v_i, \omega_i}(\mathbb{R}^n)}.
\end{align*}
Therefore, Theorem \ref{theo3.1} is completely proved.
\end{proof}
Now, we would like to give an application of Theorem \ref{theo3.1}. Let us take the matrices $A_i(y)={\rm diag }[ s_{i}(y),...,s_{i}(y) ]$, for all $i=1,...,m$, where the measurable functions $s_1(y), ..., s_m(y)\not=0$ almost everywhere in $\mathbb R^n$. It is obvious that the matrices $A_i$'s satisfy the condition (\ref{eq3.1}). 
By Theorem \ref{theo3.1},  we obtain the following useful result concerning the boundedness of the multilinear operator $\mathcal{H}_{\phi,\vec{s}}$.
\begin{corollary} \label{H-phi-Morrey}
Let $\phi$ be a nonnegative function. Under the same assumptions as Theorem \ref{theo3.1}, we have that the operator $\mathcal{H}_{\phi,\vec{s}}$ is bounded from $\prod_{i=1}^m\dot{M}^{q_i,\lambda_i}_{v_i, \omega_i}(\mathbb{R}^n)$ to $\dot{M}^{q,\lambda}_{v,\omega}(\mathbb{R}^n)$ if and only if
\begin{align*}
\mathcal{C}_{1.1}=\int_{\mathbb{R}^n}\left(\prod_{i=1}^m|s_i(y)|^{(\beta_i+n)\lambda_i+(\beta_i-\gamma_i)\frac{1}{q_i}}\right)\phi(y) dy<\infty.
\end{align*}
Moreover,
\begin{align*}
\|\mathcal{H}_{\phi,\vec{s}} \|_{\prod_{i=1}^m\dot{M}^{q_i,\lambda_i}_{v_i, \omega_i}(\mathbb{R}^n)\to \dot{M}^{q,\lambda}_{v,\omega}(\mathbb{R}^n)}\simeq \mathcal{C}_{1.1}.
\end{align*} 
\end{corollary} 
 
In particular, by virtue of Corollary \ref{H-phi-Morrey} one can claim that the weighted multilinear Hardy-Ces\`{a}ro operator $U_{\psi,\vec{s}}^{m,n}$ is bounded from $\prod_{i=1}^m\dot{M}^{q_i,\lambda_i}_{v_i, \omega_i}(\mathbb{R}^n)$ to $\dot{M}^{q,\lambda}_{v,\omega}(\mathbb{R}^n)$ if and only if
\begin{align*}
\mathcal{C}_{1.2}=\int_{[0, 1]^n}\left(\prod_{i=1}^m|s_i(t)|^{(\beta_i+n)\lambda_i+(\beta_i-\gamma_i)\frac{1}{q_i}}\right)\psi(t) dt<\infty.
\end{align*}
Moreover, 
\begin{align*}
\|U_{\psi,\vec{s}}^{m,n}\|_{\prod_{i=1}^m\dot{M}^{q_i,\lambda_i}_{v_i, \omega_i}(\mathbb{R}^n)\to \dot{M}^{q,\lambda}_{v,\omega}(\mathbb{R}^n)}\simeq \mathcal{C}_{1.2}.
\end{align*} 

Next, we also give the boundedness and bounds of the multilinear Hausdorff operator on the two weighted Herz spaces.
\begin{theorem} \label{theo3.2}
Let $\omega_i(x)=|x|^{\gamma_i}, v_i(x)=|x|^{\beta_i}, \textit{ for all } i=1,...,m, \omega(x)=|x|^\gamma, v(x)=|x|^\beta$, and the following conditions hold
$$\sum\limits_{i=1}^m \frac{\gamma_i}{q_i}=\frac{\gamma}{q}, \textit{ and }
\sum\limits_{i=1}^m \left(1+\frac{\beta_i}{n}\right)\alpha_i=\left(1+\frac{\beta}{n}\right)\alpha.
$$
{\rm(i)}  If 
\begin{align*}
\mathcal{C}_{2}=\int_{\mathbb{R}^n}\frac{|\Phi(y)|}{|y|^n}\prod_{i=1}^m\|A^{-1}_i(y)\|^{\left(1+\frac{\beta_i}{n}  \right)\alpha_i+\frac{n+\gamma_i}{q_i}}  dy<\infty,
\end{align*}
then ${\mathcal{H}}_{\Phi,\vec{A}}$ is bounded from $\prod_{i=1}^m\dot{K}^{\alpha_i,p_i,q_i}_{v_i,\omega_i}(\mathbb{R}^n)$ to $\dot{K}^{\alpha,p,q}_{v,\omega}(\mathbb{R}^n)$. Moreover,
\begin{align*}
\|{\mathcal{H}}_{\Phi,\vec{A}} \|_{\prod_{i=1}^m\dot{K}^{\alpha_i,p_i,q_i}_{v_i,\omega_i}(\mathbb{R}^n) \to \dot{K}^{\alpha,p,q}_{v,\omega}(\mathbb{R}^n)}\lesssim \mathcal{C}_{2}.
\end{align*}
(ii) Conversely, suppose $\Phi$ is a real function with a constant sign in $\mathbb{R}^n$. Then, if ${\mathcal{H}}_{\Phi,\vec{A}}$ is bounded from $\prod_{i=1}^m\dot{K}^{\alpha_i,p_i,q_i}_{v_i,\omega_i}(\mathbb{R}^n)$ to $\dot{K}^{\alpha,p,q}_{v,\omega}(\mathbb{R}^n)$, we have $\mathcal{C}_2<\infty$. Furthermore,
\begin{align*}
\|{\mathcal{H}}_{\Phi,\vec{A}} \|_{\prod_{i=1}^m\dot{K}^{\alpha_i,p_i,q_i}_{v_i,\omega_i}(\mathbb{R}^n) \to \dot{K}^{\alpha,p,q}_{v,\omega}(\mathbb{R}^n)}\gtrsim \mathcal{C}_{2}.
\end{align*}
\end{theorem}

\begin{proof}
(i) By the same arguments as the inequality \eqref{basic-ineq} in the proof of Theorem \ref{theo3.1}, we also have
\begin{align*}
&\|{\mathcal{H}}_{\Phi,\vec{A}}(\vec{f})\chi_k\|_{L^q_\omega (\mathbb{R}^n)}\\
&\le  \int_{\mathbb{R}^n}\frac{|\Phi(y)|}{|y|^n}\prod_{i=1}^m\max\{\|A_i^{-1}(y)\|^{\gamma_i},\|A_i(y)\|^{-\gamma_i} \}^{\frac{1}{q_i}}|\det A^{-1}_i(y)|^{\frac{1}{q_i}}   \|f_i\|_{L^{q_i}_{\omega_i}(A_i(y)C_k)} dy,
\end{align*}
where $A_i(y)C_k=\{A_i(y)z|z\in C_k \}$. By the condition \eqref{eq3.1}, there exits the greatest integer number $\kappa^*=\kappa^*(y)$ statisfying
\begin{align*}
\max\limits_{i=1,...,m}\{\|A_i(y)\|.\|A_i^{-1}(y)\| \}< 2^{-\kappa^*}, \text{ for a.e } y\in \mathbb{R}^n.
\end{align*}
Note that from the condition $1\leq\|A_i(y)\|.\|A^{-1}_i(y)\|\leq\rho_{\vec{A}}$, for a.e $y\in\mathbb R^n$, $i=1,...,m$, it follows that
$|\kappa^*(y)|\simeq 1$ for a.e $y\in\mathbb R^n$.
\vskip 5pt
Let us now fix $i\in \{1,2,...,m \}$. Since $\|A_i(y)\|\ne 0$, there is an integer number $\ell_i=\ell_i(y)$ such that $2^{\ell_i-1}<\|A_i(y)\|\le 2^{\ell_i}$. For simplicity of notation, we write
\begin{align*}
\rho^*_{\vec{A}}(y)=\max\limits_{i=1,...,m}\{\|A_i(y)\|.\|A_i^{-1}(y)\| \}.
\end{align*}
Then, by letting $t=A_i(y).z$, with $z\in C_k$, it follows that
\begin{align*}
|t|\ge \|A_i^{-1}(y)\|^{-1}.|z|\ge \frac{2^{k+\ell_i-2}}{\rho^*_{\vec{A}}(y)}>2^{k+\ell_i-2+\kappa^*},
\end{align*}
and
\begin{align*}
|t|\le \|A_i(y)\|.|z|\le 2^{k+\ell_i}.
\end{align*}
These estimates can be used to obtain
\begin{align*}
A_i(y)C_k\subset \{z\in\mathbb{R} ^n: 2^{k+\ell_i-2+\kappa^*}<|z|\le 2^{k+\ell_i}\},
\end{align*}
which implies that 
\begin{align}\label{eq3.5.1} 
\|{\mathcal{H}}_{\Phi,\vec{A}}(\vec{f})\chi_k\|_{L^q_\omega (\mathbb{R}^n)}
&\le  \int_{\mathbb{R}^n}\frac{|\Phi(y)|}{|y|^n}\prod_{i=1}^m\max\{\|A_i^{-1}(y)\|^{\gamma_i},\|A_i(y)\|^{-\gamma_i} \}^{\frac{1}{q_i}}|\det A^{-1}_i(y)|^{\frac{1}{q_i}} \times\notag\\
&\;\;\;\;\;\;\times \left(\sum\limits_{r=\kappa^*-1}^0 \|f_i\chi_{k+\ell_i+r}\|_{L^{q_i}_{\omega_i}(\mathbb{R}^n)}\right) dy.
\end{align}
On the other hand, by the definition of two weighted Herz space and the Minkowski inequality, we get
\begin{align*}
&\|{\mathcal{H}}_{\Phi,\vec{A}}(\vec{f}) \|_{\dot{K}^{\alpha,p,q}_{v,\omega}(\mathbb{R}^n)}\\
&=\left(\sum\limits_{k=-\infty}^\infty v(B_k)^{\frac{\alpha p}{n}}   \left| \int_{\mathbb{R}^n}\frac{|\Phi(y)|}{|y|^n}\prod_{i=1}^m\max\{\|A_i^{-1}(y)\|^{\gamma_i},\|A_i(y)\|^{-\gamma_i} \}^{\frac{1}{q_i}}|\det A^{-1}_i(y)|^{\frac{1}{q_i}}   \right.\right.\times\\
&\left.\left.\times\left(\sum\limits_{r=\kappa^*-1}^0 \|f_i\chi_{k+\ell_i+r}\|_{L^{q_i}_{\omega_i}(\mathbb{R}^n)}\right) dy\right|^p_{L^q_\omega(\mathbb{R}^n)} \right)^{\frac{1}{p}}\\
&\le \int_{\mathbb{R}^n}\frac{|\Phi(y)|}{|y|^n}\prod_{i=1}^m\max\{\|A_i^{-1}(y)\|^{\gamma_i},\|A_i(y)\|^{-\gamma_i} \}^{\frac{1}{q_i}}|\det A^{-1}_i(y)|^{\frac{1}{q_i}}\times\\
&\times      \left(\sum\limits_{k=-\infty}^\infty v(B_k)^{\frac{\alpha p}{n}}\left(\prod_{i=1}^m    \sum\limits_{r=\kappa^*-1}^0 \|f_i\chi_{k+\ell_i+r}\|_{L^{q_i}_{\omega_i}(\mathbb{R}^n)} \right)^p          \right)^{\frac{1}{p}}dy.
\end{align*}
Notice that taking the condition $\sum\limits_{i=1}^m (n+\beta_i)\alpha_i=(n+\beta)\alpha $ into account, we get $v(B_k)^\alpha\simeq \prod_{i=1}^mv_i(B_k)^{\alpha_i}$. Applying the H\"{o}lder inequality, we have
\begin{align}\label{eq3.5.2}
&\left(\sum\limits_{k=-\infty}^\infty v(B_k)^{\frac{\alpha p}{n}}\left(\prod_{i=1}^m    \sum\limits_{r=\kappa^*-1}^0 \|f_i\chi_{k+\ell_i+r}\|_{L^{q_i}_{\omega_i}(\mathbb{R}^n)} \right)^p          \right)^{\frac{1}{p}}\notag\\
&\simeq\left(\sum\limits_{k=-\infty}^\infty \left(\prod_{i=1}^m   v_i(B_k)^{\frac{\alpha_i}{n}} \sum\limits_{r=\kappa^*-1}^0 \|f_i\chi_{k+\ell_i+r}\|_{L^{q_i}_{\omega_i}(\mathbb{R}^n)} \right)^p          \right)^{\frac{1}{p}}\notag\\
&\lesssim \prod_{i=1}^m \left(\sum\limits_{k=-\infty}^\infty v_i(B_k)^{\frac{\alpha_i p_i}{n}}\left(   \sum\limits_{r=\kappa^*-1}^0 \|f_i\chi_{k+\ell_i+r}\|_{L^{q_i}_{\omega_i}(\mathbb{R}^n)} \right)^{p_i}          \right)^{\frac{1}{p_i}}.
\end{align}
Moreover, by using the known inequality $\left(\sum_{i=1}^N|a_i|\right)^p\leq N^{p-1}\sum_{i=1}^N|a_i|^p$ for all $p\ge 1$, we have, by  $p_i\ge 1$, 
\begin{align}\label{eq3.5.3}
\left(\sum\limits_{r=\kappa^*-1}^0 \|f_i\chi_{k+\ell_i+r}\|_{L^{q_i}_{\omega_i}(\mathbb{R}^n)}\right)^{p_i} \le (2-\kappa^*)^{p_i-1} \sum\limits_{r=\kappa^*-1}^0 \|f_i\chi_{k+\ell_i+r}\|_{L^{q_i}_{\omega_i}(\mathbb{R}^n)}^{p_i}.
\end{align}
Consequently, we obtain
\begin{align*}
&\|{\mathcal{H}}_{\Phi,\vec{A}}(\vec{f}) \|_{\dot{K}^{\alpha,p,q}_{v,\omega}(\mathbb{R}^n)}\\
&\lesssim \int_{\mathbb{R}^n}(2-\kappa^*)^{m-\frac{1}{p}}   \frac{|\Phi(y)|}{|y|^n}\prod_{i=1}^m\max\{\|A_i^{-1}(y)\|^{\gamma_i},\|A_i(y)\|^{-\gamma_i} \}^{\frac{1}{q_i}}|\det A^{-1}_i(y)|^{\frac{1}{q_i}}.\mathcal{H}_{1i} dy,
\end{align*}
where
\begin{align*}
\mathcal{H}_{1i}:= \prod_{i=1}^m\sum\limits_{r=\kappa^*-1}^0 \left( \sum\limits_{k=-\infty}^{\infty}  v_i(B_k)^{\frac{\alpha_i p_i}{n}}     \|f_i\chi_{k+\ell_i+r}\|_{L^{q_i}_{\omega_i}(\mathbb{R}^n)}^{p_i} \right)^{\frac{1}{p_i}}.
\end{align*}
Since $2^{\ell_i-1}<\|A_i(y) \|\le 2^{\ell_i}$, it implies that
\begin{align*}
2^{-\ell_i}\lesssim \|A_i(y)\|^{-1}\Rightarrow 2^{-\ell_i\left(1+\frac{\beta_i}{n}  \right)\alpha_i}\lesssim   \|A_i(y)\|^{-\left(1+\frac{\beta_i}{n}  \right)\alpha_i}.
\end{align*}
Also, remark that $ v_i(B_k)^{\frac{\alpha_i p_i}{n}}= 2^{-(\ell_i+r)\left(1+\frac{\beta_i}{n}  \right)\alpha_ip_i}.v_i(B_{k+\ell_i+r})^{\frac{\alpha_i p_i}{n}}$. 
Hence, it is easy to get that
\begin{align*}
\mathcal{H}_{1i}
&\le \prod_{i=1}^m \sum\limits_{r=\kappa^*-1}^0 2^{-(\ell_i+r){\left(1+\frac{\beta_i}{n}  \right)\alpha_i}}  \left( \sum\limits_{k=-\infty}^{\infty}  v_i(B_{k+\ell_i+r})^{\frac{\alpha_i p_i}{n}}     \|f_i\chi_{k+\ell_i+r}\|_{L^{q_i}_{\omega_i}(\mathbb{R}^n)}^{p_i} \right)^{\frac{1}{p_i}}\\
&\lesssim \prod_{i=1}^m \left(\sum\limits_{r=\kappa^*-1}^0 2^{-r{\left(1+\frac{\beta_i}{n}  \right)\alpha_i}} \right)\|A_i(y)\|^{-{\left(1+\frac{\beta_i}{n}  \right)\alpha_i}}\|f_i\|_{\dot{K}^{\alpha_i,p_i,q_i}_{v_i,\omega_i}(\mathbb{R}^n)}.
\end{align*}
Thus, we obtain
\begin{align*}
&\|{\mathcal{H}}_{\Phi,\vec{A}}(\vec{f}) \|_{\dot{K}^{\alpha,p,q}_{v,\omega}(\mathbb{R}^n)}\\
&\lesssim\int_{\mathbb{R}^n}(2-\kappa^*)^{m-\frac{1}{p}}\frac{|\Phi(y)|}{|y|^n}\prod_{i=1}^m\max\{\|A_i^{-1}(y)\|^{\gamma_i},\|A_i(y)\|^{-\gamma_i} \}^{\frac{1}{q_i}}|\det A^{-1}_i(y)|^{\frac{1}{q_i}}\times\\
&\;\;\;\;\;\times \|A_i(y)\|^{-{\left(1+\frac{\beta_i}{n}  \right)\alpha_i}}  \left(\sum\limits_{r=\kappa^*-1}^0 2^{-r{\left(1+\frac{\beta_i}{n}  \right)\alpha_i}} \right)dy . \prod_{i=1}^m\|f_i\|_{\dot{K}^{\alpha_i,p_i,q_i}_{v_i,\omega_i}(\mathbb{R}^n)}.
\end{align*}
Similarly to estimate for the expression \eqref{eq3.5}, we get
\begin{align*}
&\prod_{i=1}^m\max\{\|A_i^{-1}(y)\|^{\gamma_i},\|A_i(y)\|^{-\gamma_i} \}^{\frac{1}{q_i}}|\det A^{-1}_i(y)|^{\frac{1}{q_i}}\|A_i(y)\|^{-\left(1+\frac{\beta_i}{n}  \right)\alpha_i}\\
&\lesssim \prod_{i=1}^m\|A^{-1}_i(y)\|^{\frac{n}{q_i}}\|A^{-1}_i(y)\|^{\frac{\gamma_i}{q_i}}\|A_i^{-1}(y)\|^{\left(1+\frac{\beta_i}{n}  \right)\alpha_i}=\prod_{i=1}^m\|A^{-1}_i(y)\|^{\left(1+\frac{\beta_i}{n}  \right)\alpha_i+\frac{n+\gamma_i}{q_i}}.
\end{align*}
Therefore, by $\kappa^*=|\kappa^*(y)|\simeq 1$ for a.e $y\in\mathbb R^n$, we have
\begin{align*}
\|{\mathcal{H}}_{\Phi,\vec{A}}(\vec{f}) \|_{\dot{K}^{\alpha,p,q}_{v,\omega}(\mathbb{R}^n)}
\lesssim \mathcal{C}_2. \prod_{i=1}^m\|f_i\|_{\dot{K}^{\alpha_i,p_i,q_i}_{v_i,\omega_i}(\mathbb{R}^n)},
\end{align*}
which means that ${\mathcal{H}}_{\Phi,\vec{A}}$ is bounded from the product space $\prod_{i=1}^m\dot{K}^{\alpha_i,p_i,q_i}_{v_i,\omega_i}(\mathbb{R}^n)$ to $\dot{K}^{\alpha,p,q}_{v,\omega}(\mathbb{R}^n)$. The proof of part (i) is completed.
\vskip 5pt
(ii) Conversely, suppose that ${\mathcal{H}}_{\Phi,\vec{A}}$ is a bounded operator from  $\prod_{i=1}^m\dot{K}^{\alpha_i,p_i,q_i}_{v_i,\omega_i}(\mathbb{R}^n)$ to $\dot{K}^{\alpha,p,q}_{v,\omega}(\mathbb{R}^n)$. For all $i=1,...,m$, let us choose the functions $f_i$ as follows
\begin{align*}
f_i(x)=\begin{cases}0, \hspace*{3cm} &|x|<\rho_{\vec{A}}^{-1},\\ |x|^{-\left(1+\frac{\beta_i}{n} \right)\alpha_i-\frac{n}{q_i}-\frac{\gamma_i}{q_i}-\varepsilon}, & \text{otherwise}.\end{cases}
\end{align*}
It is obvious to see that for any integer number $k$ statisfying $k<-\frac{\ln(\rho_{\vec{A}})}{\ln2}$, then $\|f_i\chi_k\|_{L^{q_i}_{\omega_i}}=0$ for all $i=1,...,m$. Otherwise, one has
\begin{align*}
\|f_i\chi_k \|_{L^{q_i}_{\omega_i}(\mathbb{R}^n)}&=\left(\int_{C_k}\int_{S_{n-1}} r^{-\left(1+\frac{\beta_i}{n} \right)\alpha_iq_i-n-\gamma_i-q_i\varepsilon}r^{\gamma_i}r^{n-1}d\sigma(x')dr   \right)^{\frac{1}{q_i}}\\
&= |S_{n-1}|^{\frac{1}{q_i}}\left(\int_{C_k} r^{-\left(1+\frac{\beta_i}{n} \right)\alpha_iq_i-q_i\varepsilon-1}dr \right)^{\frac{1}{q_i}}\\
&\simeq  2^{-k\left(\varepsilon+\left(1+\frac{\beta_i}{n} \right)\alpha_i\right)}\left(\frac{2^{q_i\left(\varepsilon+\left(1+\frac{\beta_i}{n} \right)\alpha_i\right)}-1}{q_i\left(\varepsilon+\left(1+\frac{\beta_i}{n} \right)\alpha_i\right)} \right)^{\frac{1}{q_i}}.
\end{align*}
Evidently, $v_i(B_k)\simeq 2^{k(n+\beta_i)}$. Therefore, an easy computation shows that
\begin{align*}
\|f_i\|_{\dot{K}^{\alpha_i,p_i,q_i}_{v_i,\omega_i}(\mathbb{R}^n)}
&\simeq\left(\sum\limits_{k\geq \theta}^\infty v_i(B_k)^{\frac{\alpha_i p_i}{n}}\left|2^{-k\left(\varepsilon+\left(1+\frac{\beta_i}{n} \right)\alpha_i\right)}\left(\frac{2^{q_i\left(\varepsilon+\left(1+\frac{\beta_i}{n} \right)\alpha_i\right)}-1}{q_i\left(\varepsilon+\left(1+\frac{\beta_i}{n} \right)\alpha_i\right)} \right)^{\frac{1}{q_i}}\right|^{p_i} \right)^{\frac{1}{p_i}}\\
&\simeq  \left( \sum\limits_{k\geq \theta}^{\infty}2^{-k\varepsilon p_i}\right)^{\frac{1}{p_i}}\left(\frac{2^{q_i\left(\varepsilon+\left(1+\frac{\beta_i}{n} \right)\alpha_i\right)}-1}{q_i\left(\varepsilon+\left(1+\frac{\beta_i}{n} \right)\alpha_i\right)} \right)^{\frac{1}{q_i}}\\
&= \left( {\frac{{{2^{(1-\theta)\varepsilon {p_i} }}}}{{{2^{\varepsilon {p_i}}} - 1}}} \right)^{\frac{1}{{{p_i}}}}\left(\frac{2^{q_i\left(\varepsilon+\left(1+\frac{\beta_i}{n} \right)\alpha_i\right)}-1}{q_i\left(\varepsilon+\left(1+\frac{\beta_i}{n} \right)\alpha_i\right)} \right)^{\frac{1}{q_i}},
\end{align*}
where $\theta$ is the smallest integer number such that $\theta\geq\frac{-{\rm ln}(\rho_{\vec A})}{{\rm ln2}}$. Next, consider  two useful sets as follows
\begin{align*}
D_x=\bigcap\limits_{i=1}^m\{y\in\mathbb{R}^n: |A_i(y)x|\ge \rho^{-1}_{\vec{A}} \},
\end{align*}
and
\begin{align*}
E=\{y\in\mathbb{R}^n: \|A_i(y)\|\ge \varepsilon, ~ \text{for all} ~ i=1,...,m \}.
\end{align*}
It is not difficult to show that 
\begin{align}\label{eq3.6}
E\subset D_x ~ \text{for all}~x\in\mathbb{R}^n\backslash B\left(0,\varepsilon \right).
\end{align}
Indeed, let $y\in E$. It is easy to check that $\|A_i(y)\|.|x|\ge 1$ for all $x\in \mathbb{R}^n\backslash B\left(0,\varepsilon \right)$. Hence, it follows from  the condition \eqref{eq3.1} that
\begin{align*}
|A_i(y)x|\ge \| A_i^{-1}(y)\|^{-1}.|x|\ge \rho^{-1}_{\vec{A}},
\end{align*}
which implies the proof of the relation \eqref{eq3.6}. Now, by letting $y\in \mathbb{R}^n\backslash B\left(0,\varepsilon^{-1} \right)$ and using \eqref{eq3.1},  \eqref{eq3.6}, we get
\begin{align*}
{\mathcal{H}}_{\Phi,\vec{A}}(\vec{f})(x)
&\ge\int_{D_x}\frac{\Phi(y)}{|y|^n}\prod_{i=1}^m|A_i(y).x|^{-\left(1+\frac{\beta_i}{n} \right)\alpha_i-\frac{(n+\gamma_i)}{q_i}-\varepsilon}dy\\
&\ge\int_{E}\frac{\Phi(y)}{|y|^n}\prod_{i=1}^m|A_i(y).x|^{-\left(1+\frac{\beta_i}{n} \right)\alpha_i-\frac{(n+\gamma_i)}{q_i}-\varepsilon}dy\\
&\gtrsim \left(\int_{E}\frac{\Phi(y)}{|y|^n}\prod_{i=1}^m\|A_i^{-1}(y)\|^{\left(1+\frac{\beta_i}{n} \right)\alpha_i+\frac{n+\gamma_i}{q_i}+\varepsilon}dy \right). |x|^{-\left(1+\frac{\beta}{n} \right)\alpha-\frac{(n+\gamma)}{q}-m\varepsilon}\chi_{\mathbb{R}^n\backslash B\left(0,\varepsilon^{-1} \right)}.
\end{align*}
Let $k_0$ be the smallest integer number such that $2^{k_0-1}\ge \varepsilon^{-1}$. We thus have
\begin{align*}
&\| {\mathcal{H}}_{\Phi,\vec{A}}(\vec{f})\chi_k \|_{L^q_\omega(\mathbb{R}^n)}\\
&\gtrsim \left( \int_{C_k}\left|\left(\int_{E}\frac{\Phi(y)}{|y|^n}\prod_{i=1}^m                  \|A_i^{-1}(y)\|^{\left(1+\frac{\beta_i}{n} \right)\alpha_i+\frac{n+\gamma_i}{q_i}+\varepsilon}dy \right) |x|^{-\left(1+\frac{\beta}{n} \right)\alpha-\frac{(n+\gamma)}{q}-m\varepsilon}\right|^q\omega(x)dx\right)^{\frac{1}{q}}\\
&\gtrsim \left(\int_{E}\frac{|\Phi(y)|}{|y|^n}\prod_{i=1}^m                  \|A_i^{-1}(y)\|^{\left(1+\frac{\beta_i}{n} \right)\alpha_i+\frac{n+\gamma_i}{q_i}+\varepsilon}dy \right)\left( \int_{C_k} |x|^{-\left(1+\frac{\beta}{n} \right)\alpha q-n-m\varepsilon q}dx 
 \right)^{\frac{1}{q}}.
\end{align*}
Consequently, 
\begin{align}\label{HA-Herz1}
\|{\mathcal{H}}_{\Phi,\vec{A}}(\vec{f})\|_{\dot{K}^{\alpha,p,q}_{v,\omega}(\mathbb{R}^n)}
&\gtrsim\left(\sum\limits_{k=k_0}^\infty v(B_k)^{\frac{\alpha p}{n}}\left|\left(\int_{E}\frac{|\Phi(y)|}{|y|^n}\prod_{i=1}^m                  \|A_i^{-1}(y)\|^{\left(1+\frac{\beta_i}{n} \right)\alpha_i+\frac{n+\gamma_i}{q_i}+\varepsilon}dy \right)\right.\right.\times\notag\\
&\;\;\;\;\;\;\left.\left.\times  \left( \int_{C_k} |x|^{-\left(1+\frac{\beta}{n} \right)\alpha q-n-m\varepsilon q}dx 
 \right)^{\frac{1}{q}}\right|^{p} \right)^{\frac{1}{p}}\notag\\
&\gtrsim  \left(\int_{E}\frac{|\Phi(y)|}{|y|^n}\prod_{i=1}^m                  \|A_i^{-1}(y)\|^{\left(1+\frac{\beta_i}{n} \right)\alpha_i+\frac{n+\gamma_i}{q_i}+\varepsilon}dy \right)\times\notag\\
&\;\;\;\;\;\; \times\left(\sum\limits_{k=k_0}^\infty v(B_k)^{\frac{\alpha p}{n}}\left( \int_{C_k} |x|^{-\left(1+\frac{\beta}{n} \right)\alpha q-n-m\varepsilon q}dx 
 \right)^{\frac{p}{q}} \right)^{\frac{1}{p}}.
\end{align}
Observe that $v(B_k)\simeq 2^{k(n+\beta)}$, and an elementary calculation shows that
\begin{align*}
\left( \int_{C_k} |x|^{-\left(1+\frac{\beta}{n} \right)\alpha q-n-m\varepsilon q}dx 
 \right)^{\frac{p}{q}} \simeq 2^{-kp\left(m\varepsilon+\left(1+\frac{\beta}{n} \right)\alpha\right)}\left(\frac{2^{q\left(m\varepsilon+\left(1+\frac{\beta}{n} \right)\alpha\right)}-1}{q\left(m\varepsilon+\left(1+\frac{\beta}{n} \right)\alpha\right)} \right)^{\frac{p}{q}},
\end{align*}
so
\begin{align}\label{Vanh}
\sum\limits_{k=k_0}^\infty v(B_k)^{\frac{\alpha p}{n}}\left( \int_{C_k} |x|^{-\left(1+\frac{\beta}{n} \right)\alpha q-n-m\varepsilon q}dx 
 \right)^{\frac{p}{q}} \simeq 
 \left( {\frac{{{2^{ - {k_0}\varepsilon mp}}}}{{1 - {2^{ - \varepsilon mp}}}}} \right)
\left(\frac{2^{q\left(m\varepsilon+\left(1+\frac{\beta}{n} \right)\alpha\right)}-1}{q\left(m\varepsilon+\left(1+\frac{\beta}{n} \right)\alpha\right)} \right)^{\frac{p}{q}}.
\end{align}
For simplicity of exposition, we denote
\begin{align*}
\theta^*(\varepsilon)=\frac{ \left( {\frac{{{2^{ - {k_0}\varepsilon mp}}}}{{1 - {2^{ - \varepsilon mp}}}}} \right)^{\frac{1}{p}}
\left(\frac{2^{q\left(m\varepsilon+\left(1+\frac{\beta}{n} \right)\alpha\right)}-1}{q\left(m\varepsilon+\left(1+\frac{\beta}{n} \right)\alpha\right)} \right)^{\frac{1}{q}}}{\prod_{i=1}^m  \left( {\frac{{{2^{(1-\theta)\varepsilon {p_i} }}}}{{{2^{\varepsilon {p_i}}} - 1}}} \right)^{\frac{1}{{{p_i}}}}\left(\frac{2^{q_i\left(\varepsilon+\left(1+\frac{\beta_i}{n} \right)\alpha_i\right)}-1}{q_i\left(\varepsilon+\left(1+\frac{\beta_i}{n} \right)\alpha_i\right)} \right)^{\frac{1}{q_i}}}.
\end{align*}
Using $\sum_{i=1}^m \frac{1}{p_i}=\frac{1}{p}$ and $\sum_{i=1}^m \frac{1}{q_i}=\frac{1}{q}$, it is not hard to check that
$$
\lim_{\varepsilon\to 0} \varepsilon^{-m\varepsilon}\theta^*(\varepsilon)=c>0.
$$
By \eqref{HA-Herz1} and \eqref{Vanh}, it yields that
\begin{align*}
\|{\mathcal{H}}_{\Phi,\vec{A}}(\vec{f})\|_{\dot{K}^{\alpha,p,q}_{v,\omega}(\mathbb{R}^n)}
&\gtrsim  \varepsilon^{-m\varepsilon}\theta^*(\varepsilon){\prod_{i=1}^m   \|f_i\|_{\dot{K}^{\alpha_i,p_i,q_i}_{v_i,\omega_i}(\mathbb{R}^n)}}\times\\
&\times\left(\int_{E}\frac{|\Phi(y)|}{|y|^n}\prod_{i=1}^m\|A_i^{-1}(y)\|^{\left(1+\frac{\beta_i}{n} \right)\alpha_i+\frac{n+\gamma_i}{q_i}}\prod_{i=1}^m\|A_i^{-1}(y)\|^{\varepsilon}\varepsilon^{m\varepsilon} 
dy \right).
\end{align*}
Remark that $\big\|A_i(y)\big\|\geq \varepsilon$ for all $y\in E$, and by \eqref{eq3.1}, we thus have
\begin{equation}\label{prodLeb}
 \prod\limits_{i=1}^m \big\|A_i^{-1}(y)\big\|^\varepsilon.\varepsilon^{m\varepsilon}\leq \rho_{\vec A}^{\varepsilon}\lesssim 1,
\end{equation}
for $\varepsilon$  sufficiently small.
Then, letting $\varepsilon\rightarrow 0$, from assuming that ${\mathcal{H}}_{\Phi,\vec{A}}$ is  bounded  from  $\prod_{i=1}^m\dot{K}^{\alpha_i,p_i,q_i}_{v_i,\omega_i}(\mathbb{R}^n)$ to $\dot{K}^{\alpha,p,q}_{v,\omega}(\mathbb{R}^n)$, by the dominated convergence theorem of Lesbegue, we obtain
$$\int_{\mathbb{R}^n}\frac{|\Phi(y)|}{|y|^n}\prod_{i=1}^m\|A^{-1}_i(y)\|^{\left(1+\frac{\beta_i}{n}  \right)\alpha_i+\frac{n+\gamma_i}{q_i}}  dy<\infty.$$
This ends the proof of the theorem.
\end{proof}

In view of Theorem \ref{theo3.2}, we also obtain the neccessary and sufficient condition for the boundedness of the operator $\mathcal{H}_{\phi,\vec{s}}$ and  multilinear Hardy-Ces\`{a}ro operators on the two weighted Herz spaces. Namely, the following is true.
\begin{corollary}\label{Hardy-Cesaro-Herz}
Let $\phi$ be a nonnegative function. Under the same assumptions as Theorem \ref{theo3.1}, we have that the operator $\mathcal{H}_{\phi,\vec{s}}$ is bounded from  $\prod_{i=1}^m\dot{K}^{\alpha_i,p_i,q_i}_{v_i,\omega_i}(\mathbb{R}^n)$ to $\dot{K}^{\alpha,p,q}_{v,\omega}(\mathbb{R}^n)$ if and only if
\begin{align*}
\mathcal{C}_{2.1}=\int_{{\mathbb{R}}^n}\left(\prod_{i=1}^m|s_i(y)|^{-\left(1+\frac{\beta_i}{n}\right)\alpha_i-\frac{(n+\gamma_i)}{q_i}}\right)\phi(y)  dy<\infty.
\end{align*}
Moreover, 
\begin{align*}
\|\mathcal{H}_{\phi,\vec{s}}\|_{\prod_{i=1}^m\dot{K}^{\alpha_i,p_i,q_i}_{v_i,\omega_i}(\mathbb{R}^n)\to \dot{K}^{\alpha,p,q}_{v,\omega}(\mathbb{R}^n)}\simeq \mathcal{C}_{2.1}.
\end{align*}
In particular, we have that the weighted multilinear Hardy-Ces\`{a}ro operator $U_{\psi,\vec{s}}^{m,n}$ is bounded from $\prod_{i=1}^m\dot{K}^{\alpha_i,p_i,q_i}_{v_i,\omega_i}(\mathbb{R}^n)$ to $\dot{K}^{\alpha,p,q}_{v,\omega}(\mathbb{R}^n)$ if and only if
\begin{align*}
\mathcal{C}_{2.2}=\int_{{[0,1]}^n}\left(\prod_{i=1}^m|s_i(t)|^{-\left(1+\frac{\beta_i}{n}\right)\alpha_i-\frac{(n+\gamma_i)}{q_i}}\right)\psi(t)  dt<\infty.
\end{align*}
Moreover, 
\begin{align*}
\|U_{\psi,\vec{s}}^{m,n} \|_{\prod_{i=1}^m\dot{K}^{\alpha_i,p_i,q_i}_{v_i,\omega_i}(\mathbb{R}^n)\to \dot{K}^{\alpha,p,q}_{v,\omega}(\mathbb{R}^n)}\simeq \mathcal{C}_{2.2}.
\end{align*}
\end{corollary}
It should be pointed out that by the similar arguments to the proof of Theorem \ref{theo3.2}, the results of Corollary \ref{Hardy-Cesaro-Herz} are also true when the power weighed functions $v, \omega$, $v_i, \omega_i$, for $i=1,...,m$, are replaced by the approciately weights of  absolutely homogeneous type. Its proof is omitted and
 left to the reader. For further readings on the absolutely homogeneous weights, one may find in \cite{CDH2016}, \cite{CH2014} and \cite{CHH2017}.  Thus, Corollary \ref{Hardy-Cesaro-Herz} extends the results of Theorem 3.2 in  \cite{CHH2017} to two weighted setting.
\vskip 5pt
Now, let us establish the boundedness for the multilinear Hausdorff
operators on the two weighted Morrey-Herz spaces.
\begin{theorem} \label{theo3.3}
Let $\lambda_i>0, \omega(x)=|x|^\gamma, v(x)=|x|^\beta, \omega_i(x)=|x|^{\gamma_i}, v_i(x)=|x|^{\beta_i},  \textit{ for all } i=1,...,m, $, and the following conditions hold
$$\sum\limits_{i=1}^m\left(1+\frac{\beta_i}{n}\right)\lambda_i =\left(1+\frac{\beta}{n}\right)\lambda,\; \sum\limits_{i=1}^m \frac{\gamma_i}{q_i}=\frac{\gamma}{q},  \textit{ and }\sum\limits_{i=1}^m \left(1+\frac{\beta_i}{n}\right)\alpha_i=\left(1+\frac{\beta}{n}\right)\alpha.$$
{\rm(i)} If 
\begin{align*}
\mathcal{C}_{3}=\int_{\mathbb{R}^n}\frac{|\Phi(y)|}{|y|^n}\prod_{i=1}^m\|A^{-1}_i(y)\|^{\left(1+\frac{\beta_i}{n}    \right)(\alpha_i-\lambda_i)+\frac{n+\gamma_i}{q_i}}  dy<\infty,
\end{align*}
then ${\mathcal{H}}_{\Phi,\vec{A}}$ is bounded from $\prod_{i=1}^m M\dot{K}^{\alpha_i, \lambda_i, p_i,q_i}_{v_i,\omega_i}(\mathbb{R}^n)$ to $M\dot{K}^{\alpha,\lambda, p,q}_{v,\omega}(\mathbb{R}^n)$. Moreover,
\begin{align*}
\|{\mathcal{H}}_{\Phi,\vec{A}} \|_{\prod_{i=1}^m M\dot{K}^{\alpha_i, \lambda_i, p_i,q_i}_{v_i,\omega_i}(\mathbb{R}^n) \to M\dot{K}^{\alpha,\lambda,p,q}_{v,\omega}(\mathbb{R}^n)}\lesssim \mathcal{C}_{3}.
\end{align*}
{\rm(ii)} Conversely, suppose $\Phi$ is a real function with a constant sign in $\mathbb{R}^n$. Then, if ${\mathcal{H}}_{\Phi,\vec{A}}$ is a bounded operator from $\prod_{i=1}^m M\dot{K}^{\alpha_i, \lambda_i, p_i,q_i}_{v_i,\omega_i}(\mathbb{R}^n)$ to $M\dot{K}^{\alpha,\lambda, p,q}_{v,\omega}(\mathbb{R}^n)$, we have $\mathcal{C}_3<\infty$. Furthermore,
\begin{align*}
\|{\mathcal{H}}_{\Phi,\vec{A}} \|_{\prod_{i=1}^m M\dot{K}^{\alpha_i, \lambda_i, p_i,q_i}_{v_i,\omega_i}(\mathbb{R}^n) \to M\dot{K}^{\alpha, \lambda, p,q}_{v,\omega}(\mathbb{R}^n)}\gtrsim \mathcal{C}_{3}.
\end{align*}
\end{theorem}

\begin{proof}
(i) The proof  is quite similar to one of Theorem \ref{theo3.2}, but
to convenience to the readers, we also give the proof here.
By \eqref{eq3.5.1} and the  Minkowski inequality, one has
\begin{align*}
&\|{\mathcal{H}}_{\Phi,\vec{A}}(\vec{f}) \|_{M\dot{K}_{v,\omega}^{\alpha,\lambda,p,q}(\mathbb{R}^n) }\\
&\le \sup\limits_{k_0\in \mathbb{Z}}v(B_{k_0})^{-\frac{\lambda}{n}}\left(\sum\limits_{k=-\infty}^{k_0} v(B_k)^{\frac{\alpha p}{n} }  \left|\int_{\mathbb{R}^n}\frac{|\Phi(y)|}{|y|^n}\prod_{i=1}^m\max\{\|A_i^{-1}(y)\|^{\gamma_i},\|A_i(y)\|^{-\gamma_i} \}^{\frac{1}{q_i}} \right.\right.\times\\
&\;\;\;\;\;\left.\left.\times |\det A^{-1}_i(y)|^{\frac{1}{q_i}}  \left(\sum\limits_{r=\kappa^*-1}^0 \|f_i\chi_{k+\ell_i+r}\|_{L^{q_i}_{\omega_i}(\mathbb{R}^n)}\right) dy \right|^p \right)^{\frac{1}{p}}\\
&\le \sup\limits_{k_0\in \mathbb{Z}}v(B_{k_0})^{-\frac{\lambda}{n}} \int_{\mathbb{R}^n}\frac{|\Phi(y)|}{|y|^n}\prod_{i=1}^m\max\{\|A_i^{-1}(y)\|^{\gamma_i},\|A_i(y)\|^{-\gamma_i} \}^{\frac{1}{q_i}}|\det A^{-1}_i(y)|^{\frac{1}{q_i}} \times\\
&\;\;\;\;\;\times\left( \sum\limits_{k=-\infty}^{k_0}v(B_k)^{\frac{\alpha p}{n} }   \prod_{i=1}^m\left(\sum\limits_{r=\kappa^*-1}^0 \|f_i\chi_{k+\ell_i+r}\|_{L^{q_i}_{\omega_i}(\mathbb{R}^n)}\right)^p \right)^{\frac{1}{p}} dy.
\end{align*}
By the same arguments as \eqref{eq3.5.2} and \eqref{eq3.5.3}, we also obtain
\begin{align*}
&\|{\mathcal{H}}_{\Phi,\vec{A}}(\vec{f}) \|_{M\dot{K}_{v,\omega}^{\alpha,\lambda,p,q}(\mathbb{R}^n) }\\
&\lesssim \sup\limits_{k_0\in \mathbb{Z}}v(B_{k_0})^{-\frac{\lambda}{n}} \int_{\mathbb{R}^n}\frac{|\Phi(y)|}{|y|^n}\prod_{i=1}^m\max\{\|A_i^{-1}(y)\|^{\gamma_i},\|A_i(y)\|^{-\gamma_i} \}^{\frac{1}{q_i}}|\det A^{-1}_i(y)|^{\frac{1}{q_i}}  \times\\
&\times (2-\kappa^*)^{m-\frac{1}{p}}\prod_{i=1}^m \sum\limits_{r=\kappa^*-1}^0 \left( \sum\limits_{k=-\infty}^{k_0} v_i(B_k)^{\frac{\alpha_i p_i}{n} } \|f_i\chi_{k+\ell_i+r}\|_{L^{q_i}_{\omega_i}(\mathbb{R}^n)}^{p_i}      \right)^{\frac{1}{p_i}}\\
&\lesssim  \int_{\mathbb{R}^n} (2-\kappa^*)^{m-\frac{1}{p}} \frac{|\Phi(y)|}{|y|^n}\prod_{i=1}^m\max\{\|A_i^{-1}(y)\|^{\gamma_i},\|A_i(y)\|^{-\gamma_i} \}^{\frac{1}{q_i}}|\det A^{-1}_i(y)|^{\frac{1}{q_i}}  \times\\
&\times\prod_{i=1}^m \sum\limits_{r=\kappa^*-1}^0 \sup\limits_{k_0\in \mathbb{Z}}v_i(B_{k_0})^{-\frac{\lambda_i}{n}} \left( \sum\limits_{k=-\infty}^{k_0}  v_i(B_k)^{\frac{\alpha_i p_i}{n} }     \|f_i\chi_{k+\ell_i+r}\|_{L^{q_i}_{\omega_i}(\mathbb{R}^n)}^{p_i} \right)^{\frac{1}{p_i}} dy.
\end{align*}
For simplicity, set
\begin{align*}
\mathcal{H}_{2i}:= \prod_{i=1}^m\sum\limits_{r=\kappa^*-1}^0 \sup\limits_{k_0\in \mathbb{Z}}v_i(B_{k_0})^{-\frac{\lambda_i}{n}} \left( \sum\limits_{k=-\infty}^{k_0}   v_i(B_k)^{\frac{\alpha_i p_i}{n} }     \|f_i\chi_{k+\ell_i+r}\|_{L^{q_i}_{\omega_i}(\mathbb{R}^n)}^{p_i} \right)^{\frac{1}{p_i}}.
\end{align*}
By the similar estimates to $\mathcal{H}_{1i}$, we also get 
\begin{align*}
\mathcal{H}_{2i} \lesssim \prod_{i=1}^m \|A_i(y)\|^{(\lambda_i-\alpha_i) \left(1+\frac{\beta_i}{n}    \right)}.\sum\limits_{r=\kappa^*-1}^0  2^{r\left(1+\frac{\beta_i}{n}    \right)(\lambda_i-\alpha_i)} \|f_i\|_{M\dot{K}^{\alpha_i,\lambda_i,p_i,q_i}_{v_i,\omega_i}(\mathbb{R}^n) }.
\end{align*}
Consequently,
\begin{align*}
&\|{\mathcal{H}}_{\Phi,\vec{A}}(\vec{f}) \|_{M\dot{K}_{v,\omega}^{\alpha,\lambda,p,q}(\mathbb{R}^n) }\\
&\lesssim  \int_{\mathbb{R}^n} (2-\kappa^*)^{m-\frac{1}{p}} \frac{|\Phi(y)|}{|y|^n}\prod_{i=1}^m\max\{\|A_i^{-1}(y)\|^{\gamma_i},\|A_i(y)\|^{-\gamma_i} \}^{\frac{1}{q_i}}.|\det A^{-1}_i(y)|^{\frac{1}{q_i}}\times\\
&\times \|A_i(y)\|^{(\lambda_i-\alpha_i) \left(1+\frac{\beta_i}{n}    \right)}  \left(\sum\limits_{r=\kappa^*-1}^0  2^{r\left(1+\frac{\beta_i}{n}    \right)(\lambda_i-\alpha_i)}\right) dy .\prod_{i=1}^m \|f_i\|_{M\dot{K}^{\alpha_i,\lambda_i,p_i,q_i}_{v_i,\omega_i}(\mathbb{R}^n) }.
\end{align*}
It is useful to note that $\kappa^*=|\kappa^*(y)|\simeq 1$ for a.e $y\in\mathbb R^n$. By the same argument as the inequality \eqref{eq3.5}, it is clear that
\begin{align*}
&\prod_{i=1}^m\max\{\|A_i^{-1}(y)\|^{\gamma_i},\|A_i(y)\|^{-\gamma_i} \}^{\frac{1}{q_i}}|\det A^{-1}_i(y)|^{\frac{1}{q_i}}\|A_i(y)\|^{(\lambda_i-\alpha_i) \left(1+\frac{\beta_i}{n}    \right)}\\
&\lesssim\prod_{i=1}^m\|A^{-1}_i(y)\|^{\left(1+\frac{\beta_i}{n}    \right)(\alpha_i-\lambda_i)  +\frac{n+\gamma_i}{q_i}}.
\end{align*}
Hence, we obtain 
\begin{align*}
&\|{\mathcal{H}}_{\Phi,\vec{A}}(\vec{f}) \|_{M\dot{K}^{\alpha,\lambda,p,q}_{v,\omega}(\mathbb{R}^n)}\\
&\lesssim  \left(\int_{\mathbb{R}^n}\frac{|\Phi(y)|}{|y|^n}\prod_{i=1}^m\|A^{-1}_i(y)\|^{\left(1+\frac{\beta_i}{n}    \right)(\alpha_i-\lambda_i)  +\frac{n+\gamma_i}{q_i}} dy\right).\prod_{i=1}^m \|f_i\|_{M\dot{K}^{\alpha_i,\lambda_i,p_i,q_i}_{v,\omega}(\mathbb{R}^n)}\\
&\lesssim \mathcal{C}_3.\prod_{i=1}^m \|f_i\|_{M\dot{K}^{\alpha_i,\lambda_i,p_i,q_i}_{v,\omega}(\mathbb{R}^n)}.
\end{align*}
This shows that ${\mathcal{H}}_{\Phi,\vec{A}}$ is bounded from the product space $\prod_{i=1}^m M\dot{K}^{\alpha_i,\lambda_i,p_i,q_i}_{v_i,\omega_i}(\mathbb{R}^n)$ to $M\dot{K}_{v,\omega}^{\alpha,\lambda,p,q}(\mathbb{R}^n) $. Therefore, the part (i) of the theorem is proved.
\vskip 5pt
(ii) For each $i=1,...,m$, let us take
\begin{align*}
f_i(x)=|x|^{(\lambda_i-\alpha_i)\left(1+\frac{\beta_i}{n}\right)-\frac{(n+\gamma_i)}{q_i}}.
\end{align*}
It is not hard to check that
\begin{align*}
\|f_i\chi_k \|_{L^{q_i}_{\omega_i}(\mathbb{R}^n)}&=\begin{cases} \ln 2, ~ {\rm for }~(\lambda_i-\alpha_i)\left(1+\frac{\beta_i}{n}\right)=0,\\
2^{k(\lambda_i-\alpha_i)\left(1+\frac{\beta_i}{n}\right)}\left(\frac{1-2^{-q_i(\lambda_i-\alpha_i)\left(1+\frac{\beta_i}{n}\right)}}{(\lambda_i-\alpha_i)\left(1+\frac{\beta_i}{n}\right)q_i} \right)^{\frac{1}{q_i}}, ~{\rm otherwise}  .   \end{cases}
\end{align*}
Hence, we have
\begin{align*}
0<\|f_i\|_{M\dot{K}^{\alpha_i,\lambda_i,p_i,q_i}_{v_i,\omega_i}(\mathbb{R}^n) }
\simeq\sup\limits_{k_0\in \mathbb{Z}} 2^{-k_0\left(1+\frac{\beta_i}{n}\right)\lambda_i}.\left( \sum\limits_{k=-\infty}^{k_0}2^{k\left(1+\frac{\beta_i}{n}\right)\lambda_i p_i } \right)^{\frac{1}{p_i}}<\infty.
\end{align*}
From \eqref{eq3.3} and the condition $\sum\limits_{i=1}^m \left(1+\frac{\beta_i}{n}   \right)\alpha_i=\left(1+\frac{\beta}{n}  \right)\alpha$, it follows that
\begin{align*}
|A_i(y)x|^{\sum\limits_{i=1}^m \left((\lambda_i-\alpha_i)\left(1+\frac{\beta_i}{n}\right)-\frac{(n+\gamma_i)}{q_i}\right)}
\ge \prod_{i=1}^m\|A_i^{-1}(y)\|^{\left(1+\frac{\beta_i}{n}   \right)(\alpha_i-\lambda_i)+\frac{(n+\gamma_i)}{q_i}}.|x|^{(\lambda-\alpha)\left(1+\frac{\beta}{n}  \right)-\frac{(n+\gamma)}{q}}.
\end{align*}
This leads to that
\begin{align*}
{\mathcal{H}}_{\Phi,\vec{A}}(\vec{f})(x)&=\int_{\mathbb{R}^n}\frac{\Phi(y)}{|y|^n}\prod_{i=1}^m                  |A_i(y).x|^{(\lambda_i-\alpha_i)\left(1+\frac{\beta_i}{n}   \right)-\frac{(n+\gamma_i)}{q_i}}dy\\
&\gtrsim \left(\int_{\mathbb{R}^n}\frac{\Phi(y)}{|y|^n}\prod_{i=1}^m                  \|A_i^{-1}(y)\|^{\left(1+\frac{\beta_i}{n}   \right)(\alpha_i-\lambda_i)+\frac{(n+\gamma_i)}{q_i} }dy \right). |x|^{(\lambda-\alpha)\left(1+\frac{\beta}{n}  \right)-\frac{(n+\gamma)}{q}}.
\end{align*}
So,
\begin{align*}
&\| {\mathcal{H}}_{\Phi,\vec{A}}(\vec{f})(x) \|_{M\dot{K}^{\alpha_i,\lambda_i,p_i,q_i}_{v_i,\omega_i}(\mathbb{R}^n) }\\
&\gtrsim \left(\int_{\mathbb{R}^n}\frac{\Phi(y)}{|y|^n}\prod_{i=1}^m                  \|A_i^{-1}(y)\|^{\left(1+\frac{\beta_i}{n}   \right)(\alpha_i-\lambda_i)+\frac{(n+\gamma_i)}{q_i} }dy \right)\||x|^{(\lambda-\alpha)\left(1+\frac{\beta}{n}  \right)-\frac{(n+\gamma)}{q}} \|_{M\dot{K}_{v,\omega}^{\alpha,\lambda,p,q}(\mathbb{R}^n) }.
\end{align*}
Therefore, we obtain
\begin{align*}
\int_{\mathbb{R}^n}\frac{\Phi(y)}{|y|^n}\prod_{i=1}^m                  \|A_i^{-1}(y)\|^{\left(1+\frac{\beta_i}{n}   \right)(\alpha_i-\lambda_i)+\frac{(n+\gamma_i)}{q_i} }dy\lesssim\|{\mathcal{H}}_{\Phi,\vec{A}} \|_{\prod_{i=1}^mM\dot{K}^{\alpha_i,\lambda_i,p_i,q_i}_{v_i,\omega_i}(\mathbb{R}^n) \to M\dot{K}_{v,\omega}^{\alpha,\lambda,p,q}(\mathbb{R}^n) }<\infty,
\end{align*}
which finishes the proof of this theorem.
\end{proof}
By Theorem \ref{theo3.3}, we have the following useful corollary.
\begin{corollary}
Let $\phi$ be a nonnegative function. Under the same assumptions as Theorem \ref{theo3.1}, we have that the operator $\mathcal{H}_{\phi,\vec{s}}$ is bounded from  $\prod_{i=1}^mM\dot{K}^{\alpha_i,\lambda_i,p_i,q_i}_{v_i,\omega_i}(\mathbb{R}^n)$ to $M\dot{K}_{v,\omega}^{\alpha,\lambda,p,q}(\mathbb{R}^n)$
if and only if
\begin{align*}
\mathcal{C}_{3.1}=\int_{{\mathbb R}^n}\left(\prod_{i=1}^m|s_i(y)|^{(\lambda_i-\alpha_i)\left(1+\frac{\beta_i}{n}    \right)-\frac{(n+\gamma_i)}{q_i}}\right)\phi(y)  dy<\infty.
\end{align*}
Moreover,
\begin{align*}
\|\mathcal{H}_{\phi,\vec{s}}\|_{\prod_{i=1}^m M\dot{K}^{\alpha_i,p_i,q_i}_{v_i,\omega_i}(\mathbb{R}^n)\to M\dot{K}^{\alpha,p,q}_{v,\omega}(\mathbb{R}^n)}\simeq \mathcal{C}_{3.1}.
\end{align*}
\end{corollary}

In the second part of our paper, we establish some sufficient conditions for the boundedness of the operator ${\mathcal{H}}_{\Phi,\vec{A}}$  on two weighted Morrey, Herz, and Morrey-Herz spaces associated with the class of Muckenhoupt weights. It seems to be difficult to find certain necessary conditions for the boundedness of  the operator ${\mathcal{H}}_{\Phi,\vec{A}}$ on some function spaces with the  Muckenhoupt weights.
\begin{theorem}\label{Morrey-Muckenhoupt}
Let $1\le q^*, \xi, \eta<\infty, -\frac{1}{q_i}<\lambda_i<0$, for all $i=1,...,m$, and $\omega\in A_\xi, v\in A_\eta$ with the finite critical index $r_\omega, r_v$ for the reverse H\"{o}lder condition such that $\omega(B(0, R))\lesssim v(B(0, R))$ for all $R>0$. Assume that $q>q^*\xi r'_\omega, \delta_1\in (1,r_\omega), \delta_2\in (1,r_v)$, and 
 $$\lambda^*=\lambda_1+\cdots+\lambda_m,$$ 
\begin{align*}
\mathcal{C}_4= \int_{\mathbb{R}^n}\frac{|\Phi(y)|}{|y|^n}\prod_{i=1}^m |\det A_i^{-1}(y)|^{\frac{\xi}{q_i}} \|A_i(y) \| ^{\frac{\xi n}{q_i}}\mathcal{A}_{i}(y)dy<\infty,
\end{align*}
where
\begin{align*}
\mathcal{A}_{i}(y)&= \prod_{i=1}^m \left(\|A_i(y)\|^{n \left(\lambda_i+\frac{1}{q_i}\right)\frac{\delta_2-1}{\delta_2}}\chi_{\{y\in\mathbb{R}^n: \|A_i(y) \|\le 1 \} }+ \|A_i(y)\|^{n\eta  \left(\lambda_i+\frac{1}{q_i}\right)}\chi_{\{y\in\mathbb{R}^n: \|A_i(y) \|> 1 \}} \right) \times\\
&\;\;\;\;\;\;\;\;\;\;\times  
\prod_{i=1}^m \left(\|A_i(y)\|^{-\frac{n}{q_i}\frac{\delta_1-1}{\delta_1}}\chi_{\{y\in\mathbb{R}^n: \|A_i(y) \|>1 \} }+ \|A_i(y)\|^{-\frac{\xi n}{q_i}}\chi_{\{y\in\mathbb{R}^n: \|A_i(y) \|\le 1 \}} \right).
\end{align*}
Then, ${\mathcal{H}}_{\Phi,\vec{A}}$ is bounded from $\prod_{i=1}^m\dot{M}^{q_i,\lambda_i}_{v, \omega} (\mathbb{R}^n)$ to $\dot{M}^{q^*,\lambda^*}_{v, \omega} (\mathbb{R}^n)$.
\end{theorem}
\begin{proof}
By the Minkowski inequality, we have
\begin{align*}
\|{\mathcal{H}}_{\Phi,\vec{A}}(\vec{f}) \|_{L^{q^*}_\omega(B(0,R))}\le \int_{\mathbb{R}^n}\frac{|\Phi(y)|}{|y|^n}\left(\int_{B(0,R)}\prod_{i=1}^m \left|f_i(A_i(y).x)\right|^{q^*}\omega(x)dx \right)^{\frac{1}{q^*}}dy.
\end{align*}
From the condition $q>q^*\xi r'_{\omega}$, there exists $r\in (1,r_\omega)$ such that $\frac{q}{\xi}=q^*.r'$. Applying the reverse H\"{o}lder property and the H\"{o}lder inequality with $\frac{\xi}{q}=\frac{\xi}{q_1}+\cdots+\frac{\xi}{q_m}$, we get
$$
\left(\int_{B(0,R)}\prod_{i=1}^m \left|f_i(A_i(y).x)\right|^{q^*}\omega(x)dx \right)^{\frac{1}{q^*}}\le \left(\int_{B(0,R)}\prod_{i=1}^m \left|f_i(A_i(y).x)\right|^{\frac{q}{\xi}}dx \right)^{\frac{\xi}{q}}.\left(\int_{B(0,R)}\omega(x)^r dx\right)^{\frac{1}{rq^*}}$$
$$\le \left(\int_{B(0,R)}\prod_{i=1}^m \left|f_i(A_i(y).x)\right|^{\frac{q}{\xi}}dx \right)^{\frac{\xi}{q}}|B(0,R)|^{-\frac{\xi}{q}} \omega(B(0,R))^{\frac{1}{q^*}}$$
$$\le \prod_{i=1}^m \left(\int_{B(0,R)} \left|f_i(A_i(y).x)\right|^{\frac{q_i}{\xi}}dx \right)^{\frac{\xi}{q_i}}|B(0,R)|^{-\frac{\xi}{q}} \omega(B(0,R))^{\frac{1}{q^*}}.
$$
By the change of variable $z=A_i(y)x$, it follows that
\begin{align*}
&\|{\mathcal{H}}_{\Phi,\vec{A}}(\vec{f}) \|_{L^{q^*}_\omega(B(0,R))}\\
&\le\int_{\mathbb{R}^n}\frac{|\Phi(y)|}{|y|^n}.\prod_{i=1}^m|\det A^{-1}(y)|^{\frac{\xi}{q_i}}\left(\int_{A_i(y).B(0,R)} \left|f_i(z)\right|^{\frac{q_i}{\xi}}dz \right)^{\frac{\xi}{q_i}}\omega(B(0,R))^{\frac{1}{q^*}}|B(0,R)|^{-\frac{\xi}{q}}dy\\
&\le\omega(B(0,R))^{\frac{1}{q^*}}|B(0,R)|^{-\frac{\xi}{q}}\int_{\mathbb{R}^n}\frac{|\Phi(y)|}{|y|^n}\prod_{i=1}^m|\det A^{-1}(y)|^{\frac{\xi}{q_i}}\|f_i\|_{L^{\frac{q_i}{\xi}}(B(0,\|A_i(y)\|.R))}dy.
\end{align*}
According to Proposition \ref{prop2.1}, one has
\begin{align*}
\|f_i\|_{L^{\frac{q_i}{\xi}}(B(0,\|A_i(y)\|.R))}
&\lesssim |B(0,\|A_i(y)\|.R)|^{\frac{\xi}{q_i}}\left(\frac{1}{\omega(B(0,\|A_i(y)\|.R)}\int_{B(0,\|A_i(y)\|.R)} |f_i|^{q_i}\omega(x)dx \right)^{\frac{1}{q_i}}\\
&=|B(0,\|A_i(y)\|.R)|^{\frac{\xi}{q_i}}\omega(B(0,\|A_i(y)\|.R))^{-\frac{1}{q_i}}\|f_i\|_{L_\omega^{q_i}(B(0,\|A_i(y)\|.R))}.
\end{align*}
Therefore, we obtain
\begin{align*}
&\|{\mathcal{H}}_{\Phi,\vec{A}}(\vec{f}) \|_{L^{q^*}_\omega(B(0,R))}\\
&\lesssim\omega(B(0,R))^{\frac{1}{q^*}}|B(0,R)|^{-\frac{\xi}{q}}\int_{\mathbb{R}^n}\frac{|\Phi(y)|}{|y|^n}.\left(\prod_{i=1}^m|\det A_i^{-1}(y)|^{\frac{\xi}{q_i}}.|B(0,\|A_i(y)\|.R)|^{\frac{\xi}{q_i}}\right.\times\\
&\;\;\;\;\;\;\;\;\;\;\;\;\;\;\;\;\;\;\;\;\;\;\;\;\;\;\;\;\;\;\;\;\;\;\;\;\;\;\;\;\;\;\;\; \left. \times\;\omega(B(0,\|A_i(y)\|.R))^{-\frac{1}{q_i}} \|f_i\|_{L_\omega^{q_i}(B(0,\|A_i(y)\|.R))}\right) dy\\
&\lesssim \int_{\mathbb{R}^n}\frac{|\Phi(y)|}{|y|^n}.\prod_{i=1}^m \left( |\det A_i^{-1}(y)|^{\frac{\xi}{q_i}}. \|A_i(y) \| ^{\frac{\xi  n}{q_i}}\frac{1}{\omega(B(0,\|A_i(y)\|.R))^{\frac{1}{q_i}}}   \right)\times\\
&\;\;\;\;\;\;\;\;\;\;\;\;\;\;\;\;\;\;\;\;\;\;\;\;\;\;\;\;\;\;\;\;\times\left(\omega(B(0,R))^{\frac{1}{q^*}}\prod_{i=1}^m \|f_i\|_{L_\omega^{q_i}(B(0,\|A_i(y)\|.R))}\right)dy.
\end{align*}
Consequently,
\begin{align*}
&\|{\mathcal{H}}_{\Phi,\vec{A}}(\vec{f})\|_{\dot{B}^{q^*,\lambda^*}_{v,\omega} (\mathbb{R}^n)}= \sup\limits_{R>0}\frac{1}{v(B(0,R))^{\lambda^*+\frac{1}{q^*}}}\|{\mathcal{H}}_{\Phi,\vec{A}}(\vec{f}) \|_{L^{q^*}_\omega(B(0,R))}\\
&\lesssim \sup\limits_{R>0}\frac{1}{v(B(0,R))^{\lambda^*+\frac{1}{q^*}}}\int_{\mathbb{R}^n}\frac{|\Phi(y)|}{|y|^n}\prod_{i=1}^m \left( |\det A_i^{-1}(y)|^{\frac{\xi}{q_i}} \|A_i(y) \| ^{\frac{\xi  n}{q_i}}\frac{1}{\omega(B(0,\|A_i(y)\|R))^{\frac{1}{q_i}}}   \right)\times\\
&\;\;\;\;\;\;\;\;\;\;\;\;\;\;\;\;\;\;\;\;\;\;\;\;\;\;\;\;\;\;\;\;\;\;\;\;\;\;\;\;\;\;\;\;\;\;\;\;\;\;\;\times\left(\omega(B(0,R))^{\frac{1}{q^*}}
 \prod_{i=1}^m \|f_i\|_{L_\omega^{q_i}(B(0,\|A_i(y)\|.R))}\right)dy\\
&\lesssim \sup\limits_{R>0}\int_{\mathbb{R}^n}\frac{|\Phi(y)|}{|y|^n}\prod_{i=1}^m |\det A_i^{-1}(y)|^{\frac{\xi}{q_i}} \|A_i(y) \| ^{\frac{\xi  n}{q_i}} \mathcal{A}_{i}(R).\prod_{i=1}^m \|f_i\|_{\dot{M}^{q_i,\lambda_i}_{v,\omega} (\mathbb{R}^n)}dy,
\end{align*}
where
\begin{align*}
\mathcal{A}_{i}(R)=\left(\frac{\omega(B(0,R))^{\frac{1}{q^*}}}{v(B(0,R))^{\lambda^*+\frac{1}{q^*}}}\right)\left(\prod_{i=1}^m\frac{v(B(0,\|A_i(y)\|.R))^{\lambda_i+\frac{1}{q_i}}}{\omega(B(0,\|A_i(y)\|.R))^{\frac{1}{q_i}}}\right).
\end{align*}
Observe that $\frac{1}{q^*}>\frac{1}{q}$, then the condition $\omega(B(0, R))\lesssim v(B(0, R))$ implies that
 $$\left(\frac{\omega(B(0, R))}{v(B(0, R))}\right)^{\frac{1}{q^*}}\lesssim \left(\frac{\omega(B(0, R))}{v(B(0, R))}\right)^{\frac{1}{q}}.$$
Hence, by the conditions $\lambda^*=\sum_{i=1}^m\lambda_i, \frac{1}{q}=\sum_{i=1}^m\frac{1}{q_i}$, we have 
\begin{align*}
\mathcal{A}_{i}(R)
&\lesssim \prod_{i=1}^m \left(\frac{v(B(0,\|A_i(y)\|.R))}{\omega(B(0,\|A_i(y)\|.R))} \right)^{\frac{1}{q_i}} \prod_{i=1}^m \left(\frac{v(B(0,\|A_i(y)\|.R))}{v(B(0,R))}  \right)^{\lambda_i} \prod_{i=1}^m \left(\frac{\omega(B(0,R))}{v(B(0,R))} \right)^{\frac{1}{q_i}}\\
&= \prod_{i=1}^m  \left(\frac{v(B(0,\|A_i(y)\|.R))}{v(B(0,R))} \right)^{\lambda_i+\frac{1}{q_i}} \prod_{i=1}^m \left(\frac{\omega(B(0,R))}{\omega(B(0,\|A_i(y)\|.R))} \right)^{\frac{1}{q_i}}.
\end{align*}
Now, using the assumptions $\omega\in A_\xi, v\in A_\eta, \delta_1\in (1,r_\omega), \delta_2\in (1,r_v), \frac{1}{q_i}+\lambda_i>0$, and by Proposition \ref{prop2.2}, we consider the following two cases.
\vskip 5pt
{Case 1:} $0<\|A_i(y)\|\le 1$. Then, we have
$$
\left(\frac{v(B(0,\|A_i(y)\|.R))}{v(B(0,R)} \right)^{\lambda_i+\frac{1}{q_i}}\lesssim \left(\frac{|B(0,\|A_i(y)\|.R)|}{|B(0,R)|} \right)^{\left(\lambda_i+\frac{1}{q_i}\right)\frac{\delta_2-1}{\delta_2}}=\|A_i(y)\|^{n \left(\lambda_i+\frac{1}{q_i}\right)\frac{\delta_2-1}{\delta_2}},$$
$$
\left(\frac{\omega(B(0,R))}{\omega(B(0,\|A_i(y)\|.R))} \right)^{\frac{1}{q_i}} \lesssim \left(\frac{|B(0,R)|}{|B(0,\|A_i(y)\|.R)|} \right)^{\frac{\xi}{q_i}}=\|A_i(y)\|^{-\frac{n\xi}{q_i}}.
$$

Case 2. $\|A_i(y)\|> 1$. We also get
$$
\left(\frac{v(B(0,\|A_i(y)\|.R))}{v(B(0,R)} \right)^{\lambda_i+\frac{1}{q_i}}\lesssim \left(\frac{|B(0,\|A_i(y)\|.R)|}{|B(0,R)|} \right)^{\left(\lambda_i+\frac{1}{q_i}\right)\eta}=\|A_i(y)\|^{n\eta \left(\lambda_i+\frac{1}{q_i}\right)}$$
$$
\left(\frac{\omega(B(0,R))}{\omega(B(0,\|A_i(y)\|.R))} \right)^{\frac{1}{q_i}} \lesssim \left(\frac{|B(0,R)|}{|B(0,\|A_i(y)\|.R)|} \right)^{\frac{1 }{q_i} \left(\frac{\delta_1-1}{\delta_1}\right)}=\|A_i(y)\|^{-\frac{n}{q_i}\left(\frac{\delta_1-1}{\delta_1}\right)}.
$$
So, we obtain
\begin{align*}
\mathcal{A}_{i}(R)&\lesssim \prod_{i=1}^m \left(\|A_i(y)\|^{n \left(\lambda_i+\frac{1}{q_i}\right)\frac{\delta_2-1}{\delta_2}}\chi_{\{y\in\mathbb{R}^n: \|A_i(y) \|\le 1 \} }+ \|A_i(y)\|^{n\eta  \left(\lambda_i+\frac{1}{q_i}\right)}\chi_{\{y\in\mathbb{R}^n: \|A_i(y) \|> 1 \}} \right) \times\\
&\times  
\prod_{i=1}^m \left(\|A_i(y)\|^{-\frac{n}{q_i}\frac{\delta_1-1}{\delta_1}}\chi_{\{y\in\mathbb{R}^n: \|A_i(y) \|> 1 \} }+ \|A_i(y)\|^{-\frac{\xi n}{q_i}}\chi_{\{y\in\mathbb{R}^n: \|A_i(y) \|\le 1 \}} \right).
\end{align*}
Thus, we obtain
\begin{align*}
\|{\mathcal{H}}_{\Phi,\vec{A}}(\vec{f})\|_{\dot{M}^{q^*,\lambda^*}_{v,\omega} (\mathbb{R}^n)}\lesssim \left(\int_{\mathbb{R}^n}\frac{|\Phi(y)|}{|y|^n}\prod_{i=1}^m |\det A_i^{-1}(y)|^{\frac{\xi}{q_i}} \|A_i(y) \| ^{\frac{\xi n}{q_i}}\mathcal{A}_{i}(y)dy\right)\prod_{i=1}^m \|f_i\|_{\dot{M}^{q_i,\lambda_i}_{v,\omega} (\mathbb{R}^n)}.
\end{align*}
Therefore, Theorem \ref{Morrey-Muckenhoupt} is completely proved.
\end{proof}
\vskip 5pt
Finally, we also obtain the following useful result concerning the boundedness of ${\mathcal{H}}_{\Phi,\vec{A}}$ on the two weighted Morrey-Herz spaces associated with the Muckenhoupt weights.
\begin{theorem}\label{Morrey-Herz-Muckenhoupt}
Let $1\le q^*, \xi, \eta<\infty, \alpha_i<0$, $\lambda_i\geq 0$, for all $i=1,...,m$, and $\omega\in A_\xi, v\in A_\eta$ with the finite critical index $r_\omega, r_v$ for the reverse H\"{o}lder condition such that $\omega(B_k)\lesssim v(B_k)$, for every $k\in\mathbb{Z}$. Assume that $q>\textit{max }\{mq^*, q^*\xi r'_\omega\}, \delta_1\in (1,r_\omega), \delta_2\in (1,r_v)$ and $\alpha^*, \lambda^* $ are two real numbers such that
$$\lambda^*=\lambda_1+\cdots+\lambda_m \textit{ and  } \;\frac{1}{m}\left(\frac{\alpha^*}{n} +\frac{1}{q^*}\right)=\frac{\alpha_i}{n} +\frac{1}{q_i},\; i=1,...,m.$$
If $\frac{\alpha^*}{n}+ \frac{1}{q^*}\le 0$ and
\begin{align*}
\mathcal{C}_{5.1} = \prod_{i=1}^m \left(\int_{\mathbb{R}^n}\frac{|\Phi(y)|}{|y|^n}   |\det A_i^{-1}(y)|^{\frac{m\xi}{q_i}}\|A_i(y) \| ^{\frac{m\xi n}{q_i}}.\mathcal{B}_{1i}(y)dy\right)^{\frac{1}{m}}<\infty,
\end{align*}
where
\begin{align*}
\mathcal{B}_{1i}(y)&=\|A_i(y) \|^{-\left(\left(\frac{n}{q^*}+\alpha^*\right)\frac{\delta_1-1}{\delta_1}+(\alpha^*-m\lambda_i)\frac{\delta_2-1}{\delta_2}-\xi\alpha^*\right)}\chi_{ \{y\in \mathbb{R}^n:\|A_i(y)\|<1\}} +\\
&\;\;\;\;+\|A_i(y) \|^{-\left(\left(\frac{n}{q^*}+\alpha^*\right)\xi+\eta(\alpha^*-m\lambda_i)-\alpha^*\frac{\delta_1-1}{\delta_1}\right)}\chi_{ \{y\in \mathbb{R}^n:\|A_i(y)\|\ge 1\}},
\end{align*}
or $\frac{\alpha^*}{n}+ \frac{1}{q^*} > 0$ and 
\begin{align*}
\mathcal{C}_{5.2} =  \prod_{i=1}^m \left(\int_{\mathbb{R}^n}\frac{|\Phi(y)|}{|y|^n}   |\det A_i^{-1}(y)|^{\frac{m\xi}{q_i}}\|A_i(y) \| ^{\frac{m\xi n}{q_i}}.\mathcal{B}_{2i}(y)dy\right)^{\frac{1}{m}}<\infty,
\end{align*}
where
\begin{align*}
\mathcal{B}_{2i}(y)&=\|A_i(y) \|^{-\left(\frac{\xi n}{q^*}+(\alpha^*-m\lambda_i)\frac{\delta_2-1}{\delta_2}\right)}\chi_{ \{y\in \mathbb{R}^n:\|A_i(y)\|<1\}}\; +\\
&\;\;\;\;+\|A_i(y) \|^{-\left(\frac{n}{q^*}\frac{\delta_1-1}{\delta_1}+\eta(\alpha^*-m\lambda_i)\right)}\chi_{ \{y\in \mathbb{R}^n:\|A_i(y)\|\ge 1\}},
\end{align*}
then we have ${\mathcal{H}}_{\Phi,\vec{A}}$ is bounded from $\prod_{i=1}^m  M\dot{K}_{v,\omega}^{\alpha_i, \lambda_i, p_i,q_i}(\mathbb{R}^n)$ to $M\dot{K}_{v,\omega}^{\alpha^*,\lambda^*, p,q^*}(\mathbb{R}^n)$.
\end{theorem}
\begin{proof}
By the same arguments as Theorem \ref{Morrey-Muckenhoupt}, we also get
\begin{align*}
\|{\mathcal{H}}_{\Phi,\vec{A}}(\vec{f})\chi_k \|_{L^{q^*}_\omega(\mathbb{R}^n)}\le \int_{\mathbb{R}^n}\frac{|\Phi(y)|}{|y|^n}\left(\int_{C_k}\prod_{i=1}^m \left|f_i(A_i(y).x)\right|^{q^*}\omega(x)dx \right)^{\frac{1}{q^*}}dy,
\end{align*}
and 
\begin{align*}
\left(\int_{C_k}\prod_{i=1}^m \left|f_i(A_i(y).x)\right|^{q^*}\omega(x)dx \right)^{\frac{1}{q^*}}\lesssim \left(\int_{C_k}\prod_{i=1}^m \left|f_i(A_i(y).x)\right|^{\frac{q}{\xi}}dx \right)^{\frac{\xi}{q}}|B_k|^{-\frac{\xi}{q}}\omega(B_k)^{\frac{1}{q^*}}.
\end{align*}
Combining the H\"{o}lder inequality again with variable transformation $z=A_i(y)x$, we have
\begin{align*}
\|{\mathcal{H}}_{\Phi,\vec{A}}(\vec{f})\chi_k \|_{L^{q^*}_\omega(\mathbb{R}^n)}
&\lesssim  |B_k|^{-\frac{\xi}{q}}\omega(B_k)^{\frac{1}{q^*}}\int_{\mathbb{R}^n}\frac{|\Phi(y)|}{|y|^n}\prod_{i=1}^m|\det A_i^{-1}(y)|^{\frac{\xi}{q_i}}\left(\int_{A_i(y).C_k} \left|f_i(z)\right|^{\frac{q_i}{\xi}}dz \right)^{\frac{\xi}{q_i}}dy\\
&\lesssim  |B_k|^{-\frac{\xi}{q}}\omega(B_k)^{\frac{1}{q^*}}\int_{\mathbb{R}^n}\frac{|\Phi(y)|}{|y|^n}\prod_{i=1}^m|\det A_i^{-1}(y)|^{\frac{\xi}{q_i}}\|f_i\|_{L^{\frac{q_i}{\xi}}(B(0,\|A_i(y)\|.2^k))}dy.
\end{align*}
It follows from Proposition \ref{prop2.1} that
\begin{align*}
\|f_i\|_{L^{\frac{q_i}{\xi}}(B(0,\|A_i(y)\|.2^k))}
&\lesssim|B(0,\|A_i(y)\|.2^k)|^{\frac{\xi}{q_i}}\omega(B(0,\|A_i(y)\|.2^k))^{-\frac{1}{q_i}}\|f_i\|_{L_\omega^{q_i}(B(0,\|A_i(y)\|.2^k))}.
\end{align*}
Hence,  we obtain
\begin{align}\label{eq3.7}
&\|{\mathcal{H}}_{\Phi,\vec{A}}(\vec{f})\chi_k \|_{L^{q^*}_\omega(\mathbb{R}^n)}
\lesssim  |B_k|^{-\frac{\xi}{q}}.\omega(B_k)^{\frac{1}{q^*}}\int_{\mathbb{R}^n}\frac{|\Phi(y)|}{|y|^n}\left(\prod_{i=1}^m|\det A_i^{-1}(y)|^{\frac{\xi}{q_i}}|B(0,\|A_i(y)\|.2^k)|^{\frac{\xi}{q_i}}\right.\times\notag\\
&\;\;\;\;\;\;\left.\times \;\omega(B(0,\|A_i(y)\|.2^k))^{-\frac{1}{q_i}}\|f_i\|_{L_\omega^{q_i}(B(0,\|A_i(y)\|.2^k))}\right)dy\notag\\
&\lesssim \int_{\mathbb{R}^n}\frac{|\Phi(y)|}{|y|^n}\prod_{i=1}^m \left( |\det A_i^{-1}(y)|^{\frac{\xi}{q_i}} \|A_i(y) \| ^{\frac{\xi n}{q_i}}\frac{\omega(B_k)^{\frac{1}{mq^*}}}{\omega(B(0,\|A_i(y)\|.2^k))^{\frac{1}{q_i}}}\|f_i\|_{L_\omega^{q_i}(B(0,\|A_i(y)\|.2^k))}   \right)dy\notag\\
\end{align}
Thus, by $\lambda^*=\lambda_1+\cdots+\lambda_m$, 
the Minkowski and the H\"{o}lder inequalities, we have
\begin{align*}
&\|{\mathcal{H}}_{\Phi,\vec{A}}(\vec{f})\|_{M\dot{K}_{v,\omega}^{\alpha^*,\lambda^*,  p,q^*}(\mathbb{R}^n)}
\lesssim \int_{\mathbb{R}^n}\frac{|\Phi(y)|}{|y|^n}\prod_{i=1}^m \left( |\det A_i^{-1}(y)|^{\frac{\xi}{q_i}}\|A_i(y) \| ^{\frac{\xi n}{q_i}}\right) \times\\
&\times \sup_{k_0\in\mathbb Z}v(B_{k_0})^{-\frac{\lambda^*}{n}}
 \left(\sum\limits_{k=-\infty}^{k_0} \prod_{i=1}^m \frac{v(B_k)^{\frac{\alpha^* p}{mn}}\omega(B_k)^{\frac{p}{mq^*}}}{\omega(B(0,\|A_i(y)\|.2^k))^{\frac{p}{q_i}}}                     \|f_i\|^p_{L_\omega^{q_i}(B(0,\|A_i(y)\|.2^k))} \right)^{\frac{1}{p}}dy\\
&\lesssim \int_{\mathbb{R}^n}\frac{|\Phi(y)|}{|y|^n}\prod_{i=1}^m \left( |\det A_i^{-1}(y)|^{\frac{\xi}{q_i}}\|A_i(y) \| ^{\frac{\xi n}{q_i}}\right) \times\\
&\;\;\;\;\;\times \prod_{i=1}^m \sup_{k_0\in\mathbb Z}v(B_{k_0})^{-\frac{\lambda_i}{n}}
\left(\sum\limits_{k=-\infty}^{k_0} \frac{v(B_k)^{\frac{\alpha^* p_i}{mn}  }\omega(B_k)^{\frac{p_i}{mq^*}}}{\omega(B(0,\|A_i(y)\|.2^k))^{\frac{p_i}{q_i}}} \|f_i\|^{p_i}_{L_\omega^{q_i}(B(0,\|A_i(y)\|.2^k))} \right)^{\frac{1}{p_i}}.
\end{align*}
Consequently, 
\begin{align*}
\|{\mathcal{H}}_{\Phi,\vec{A}}(\vec{f})\|_{M\dot{K}_{v,\omega}^{m\alpha^*,\lambda^* p,q^*}(\mathbb{R}^n)}
&\lesssim \int_{\mathbb{R}^n}\frac{|\Phi(y)|}{|y|^n}\prod_{i=1}^m \left( |\det A_i^{-1}(y)|^{\frac{\xi}{q_i}}\|A_i(y) \| ^{\frac{\xi n}{q_i}}\mathcal{C}_{1i}(y) \right)dy,
\end{align*}
where
\begin{align*}
\mathcal{C}_{1i}(y):=\sup_{k_0\in\mathbb Z}v(B_{k_0})^{-\frac{\lambda_i}{n}}\left(\sum\limits_{k=-\infty}^{k_0}\left( \frac{v(B_k)^{\frac{\alpha^* }{mn}   }\omega(B_k)^{\frac{1}{mq^*}}}{\omega(B(0,\|A_i(y)\|.2^k))^{\frac{1}{q_i}}}  \|f_i\|_{L_\omega^{q_i}(B(0,\|A_i(y)\|.2^k))}\right)^{p_i} \right)^{\frac{1}{p_i}}.
\end{align*}
It follows readily from the H\"{o}lder inequality for $\frac{1}{m}+\cdots+\frac{1}{m}=1$ that 
\begin{align*}
\|{\mathcal{H}}_{\Phi,\vec{A}}(\vec{f})\|_{M\dot{K}_{v,\omega}^{\alpha^*, \lambda^*, p,q^*}(\mathbb{R}^n)}
&\lesssim  \prod_{i=1}^m  \left(  \int_{\mathbb{R}^n}\frac{|\Phi(y)|}{|y|^n}   |\det A_i^{-1}(y)|^{\frac{m\xi}{q_i}}\|A_i(y) \| ^{\frac{m\xi n}{q_i}}   \mathcal{C}_{1i}(y)^mdy     \right)^{\frac{1}{m}}\\
&:=\prod_{i=1}^m \mathcal{E}_{1i}^{\frac{1}{m}}.
\end{align*}
Fix $i\in\{1,...,m\}$. Since $\|A_i(y)\|\not = 0$, there is $j=j(i,y)$ such that $2^{j-1}\leq\|A_i(y)\|<2^j$. Thus, $B_{k+j-1}\subseteq B(0, \|A_i(y)\|.2^k)\subset B_{k+j}$. It implies that $\omega(B_{k+j-1})\le    \omega(B(0,\|A_i(y)\|.2^k) )$ and 
$B(0, \|A_i(y)\|.2^k)\subset\cup_{\ell=-\infty}^j C_{k+\ell}$. Combining these and using the  inequality $\left(\sum|a_\ell|\right)^\theta\leq \sum|a_\ell|^\theta$ for all $0<\theta\leq1$, it yields
\begin{align*}
\mathcal{C}_{1i}(y)
&\le\sup_{k_0\in\mathbb Z}v(B_{k_0})^{-\frac{\lambda_i}{n}}\left(\sum\limits_{k=-\infty}^{k_0}\left( \frac{v(B_k)^{\frac{\alpha^* }{mn}   }\omega(B_k)^{\frac{1}{mq^*}}}{\omega(B(0,\|A_i(y)\|.2^k))^{\frac{1}{q_i}}}  \|f_i\|_{L_\omega^{q_i}(B_{k+j})}\right)^{p_i} \right)^{\frac{1}{p_i}}.\\
&\le \sup_{k_0\in\mathbb Z}v(B_{k_0})^{-\frac{\lambda_i}{n}}\left(\sum\limits_{k=-\infty}^{k_0}  \left(  \frac{v(B_k)^{\frac{\alpha^* }{mn}   }.\omega(B_k)^{\frac{1}{mq^*}}}{\omega(B(0,\|A_i(y)\|.2^k))^{\frac{1}{q_i}}.v(B_{k+j})^{\frac{\alpha_i}{n} }}\right.\right.\times\\
&\left.\left.\;\;\;\;\;\;\;\;\;\;\;\;\;\;\;\;\;\;\;\times \sum\limits_{\ell=-\infty}^j \left(\frac{v(B_{k+j})}{v(B_{k+\ell})}  \right)^{\frac{\alpha_i}{n} }        v(B_{k+\ell})^{\frac{\alpha_i}{n} }\|f_i\|_{L_\omega^{q_i}(C_{k+\ell})}   \right)^{p_i} \right)^{\frac{1}{p_i}}\\
&\le\sup_{k_0\in\mathbb Z}v(B_{k_0})^{-\frac{\lambda_i}{n}}\left(\sum\limits_{k=-\infty}^{k_0}  \left( \mathcal{D}_k^j \sum\limits_{\ell=-\infty}^j \left(\frac{v(B_{k+j})}{v(B_{k+\ell})}  \right)^{\frac{\alpha_i}{n} }        v(B_{k+\ell})^{\frac{\alpha_i}{n} }\|f_i\|_{L_\omega^{q_i}(C_{k+\ell})}   \right)^{p_i} \right)^{\frac{1}{p_i}},
\end{align*}
where
\begin{align*}
\mathcal{D}_k^j:=\frac{v(B_k)^{\frac{\alpha^* }{mn}   }\omega(B_k)^{\frac{1}{mq^*}}}{\omega(B_{k+j-1})^{\frac{1}{q_i}}v(B_{k+j})^{\frac{\alpha_i}{n} }}.
\end{align*}
By Proposition \ref{prop2.2}, $\ell\leq j$ and $\alpha_i<0$, we have
\begin{align}\label{eq3.8}
\left(\frac{v(B_{k+j})}{v(B_{k+\ell})}  \right)^{\frac{\alpha_i}{n} }\lesssim \left(\frac{|B_{k+j}|}{|B_{k+\ell}|} \right)^{\frac{\alpha_i}{n} \frac{\delta_2-1}{\delta_2}} =2^{(j-\ell)\alpha_i\frac{\delta_2-1}{\delta_2}}.
\end{align}
Observe that
\begin{align*}
\mathcal{D}_k^j&= \frac{v(B_k)^{\frac{\alpha^* }{mn}   }\omega(B_k)^{\frac{1}{mq^*}}}{\omega(B_{k+j-1})^{\frac{1}{q_i}}v(B_{k+j})^{\frac{\alpha_i}{n} }}
=\left(\frac{\omega(B_{k+j})}{\omega(B_{k+j-1})} \right)^{\frac{1}{q_i}}\frac{v(B_k)^{\frac{\alpha^* }{mn}   }\omega(B_k)^{\frac{1}{mq^*}}}{\omega(B_{k+j})^{\frac{1}{q_i}}v(B_{k+j})^{\frac{\alpha_i}{n} }}\\
&\lesssim \left(\frac{|B_{k+j}|}{|B_{k+j-1}|} \right)^{\frac{\xi}{q_i}}\frac{v(B_k)^{\frac{\alpha^* }{mn}   }\omega(B_k)^{\frac{1}{mq^*}}}{\omega(B_{k+j})^{\frac{1}{q_i}}v(B_{k+j})^{\frac{\alpha_i}{n} }}
\lesssim \frac{v(B_k)^{\frac{\alpha^* }{mn}   }\omega(B_k)^{\frac{1}{mq^*}}}{\omega(B_{k+j})^{\frac{1}{q_i}}v(B_{k+j})^{\frac{\alpha_i}{n} }} \\
&=\frac{\omega(B_k)^{\frac{1}{mq^*}+\frac{\alpha^*}{mn} }}{\omega(B_{k+j})^{\frac{1}{q_i}+\frac{\alpha_i}{n} }}\frac{\omega(B_{k+j})^{\frac{\alpha_i}{n} }}{\omega(B_k)^{\frac{\alpha^*}{mn} }}\frac{v(B_k)^{\frac{\alpha^*}{mn} }}{v(B_{k+j})^{\frac{\alpha_i}{n} }}\\
&=\left(\frac{\omega(B_k)}{\omega(B_{k+j})} \right)^{\frac{1}{mq^*}+\frac{\alpha^*}{mn} }\left( \frac{\omega(B_{k+j})}{v(B_{k+j})} \right)^{\frac{\alpha
_i}{n}}\left(\frac{v(B_k)}{\omega(B_k)} \right)^{\frac{\alpha^*}{mn} }
\end{align*}
Observe that from the conditions $\frac{\alpha^*}{mn} +\frac{1}{mq^*}=\frac{\alpha_i}{n} +\frac{1}{q_i}$ and $q>mq^*$, it follows immediately that $\frac{1}{mq^*}>\frac{1}{q}>\frac{1}{q_i}$ and $\frac{\alpha^*}{mn}<\frac{\alpha_i}{n}$, for all $i=1,...,m$. Thus, the condition $\omega(B_k)\lesssim v(B_k)$, for every $k\in\mathbb{Z}$, implies that
$$\left( \frac{\omega(B_{k+j})}{v(B_{k+j})} \right)^{\frac{\alpha
_i}{n}}\lesssim \left( \frac{\omega(B_{k+j})}{v(B_{k+j})} \right)^{\frac{\alpha
^*}{mn}}. $$
Therefore, we obtain
\begin{align*}
\mathcal{D}_k^j 
\lesssim \left(\frac{\omega(B_k)}{\omega(B_{k+j})} \right)^{\frac{1}{mq^*}+\frac{\alpha^*}{mn} }\left( \frac{\omega(B_{k+j})}{\omega(B_{k})} \right)^{\frac{\alpha^*}{mn}}\left(\frac{v(B_k)}{v(B_{k+j})} \right)^{\frac{\alpha^*}{mn} }
\end{align*}
Now, let us consider two cases as follows.
\vskip 5pt
\textbf{Case 1:} $\frac{\alpha^*}{n}+\frac{1}{q^*}\leq 0$.\\
For $j\ge 1$, we get
\begin{align*}
\left(\frac{\omega(B_k)}{\omega(B_{k+j})} \right)^{\frac{1}{mq^*}+\frac{\alpha^*}{mn} }&\lesssim \left( \frac{|B_k|}{|B_{k+j}|}\right)^{\left(\frac{1}{mq^*}+\frac{\alpha^*}{mn}  \right)\xi}=2^{-j\xi n \left(\frac{1}{mq^*}+\frac{\alpha^*}{mn}  \right)}\\
\left( \frac{\omega(B_{k+j})}{\omega(B_{k})} \right)^{\frac{\alpha
^*}{mn}}&\lesssim \left( \frac{|B_{k+j}|}{|B_{k}|}\right)^{\frac{\alpha
^*}{mn}.\frac{\delta_1-1}{\delta_1}}=2^{j n\frac{\alpha^*}{mn}.\frac{\delta_1-1}{\delta_1}}=2^{j\frac{\alpha^*}{m}.\frac{\delta_1-1}{\delta_1}}\\
\left(\frac{v(B_k)}{v(B_{k+j})} \right)^{\frac{\alpha^*}{mn} }&\lesssim  \left( \frac{|B_{k}|}{|B_{k+j}|}\right)^{\frac{\alpha^*}{mn}  \eta}=2^{-j \eta n.\frac{\alpha^*}{mn} }=2^{-j\eta\frac{\alpha^*}{m}}.
\end{align*}
For $j\le 0$, we have
\begin{align*}
\left(\frac{\omega(B_k)}{\omega(B_{k+j})} \right)^{\frac{1}{mq^*}+\frac{\alpha^*}{mn} }&\lesssim \left( \frac{|B_k|}{|B_{k+j}|}\right)^{\left(\frac{1}{mq^*}+\frac{\alpha^*}{mn}  \right).\frac{\delta_1-1}{\delta_1}}=2^{-jn\left(\frac{1}{mq^*}+\frac{\alpha^*}{mn}  \right).\frac{\delta_1-1}{\delta_1}}\\
\left( \frac{\omega(B_{k+j})}{\omega(B_{k})} \right)^{\frac{\alpha
^*}{mn}}&\lesssim \left( \frac{|B_{k+j}|}{|B_{k}|}\right)^{\frac{\alpha
^*}{mn}.\xi}=2^{j\xi\frac{\alpha^*}{m}}\\
\left(\frac{v(B_k)}{v(B_{k+j})} \right)^{\frac{\alpha^*}{mn} }&\lesssim  \left( \frac{|B_{k}|}{|B_{k+j}|}\right)^{\frac{\alpha^*}{mn}  .\frac{\delta_2-1}{\delta_2}}=2^{-j\frac{\alpha^*}{m}\frac{\delta_2-1}{\delta_2}}.
\end{align*}
So, we obtain
\begin{align*}
\mathcal{D}_k^j \lesssim \begin{cases} 2^{\frac{1}{m}\left(-\frac{jn}{q^*}.\frac{\delta_1-1}{\delta_1}-j\alpha^*.\frac{\delta_1-1}{\delta_1}-j\alpha^*.\frac{\delta_2-1}{\delta_2}+j\xi\alpha^*\right)}, &j\le 0\\
2^{\frac{1}{m}\left(-j\xi\alpha^*-j\eta\alpha^*-\frac{j\xi n}{q^*}+j\alpha^*.\frac{\delta_1-1}{\delta_1}\right)},\qquad &j\ge 1. \end{cases}
\end{align*}

\textbf{Case 2:} $\frac{\alpha^*}{n}+\frac{1}{q^*} >0$.\\
Similarly to Case 1, for $j\ge 1$, we also have
\begin{align*}
\left(\frac{\omega(B_k)}{\omega(B_{k+j})} \right)^{\frac{1}{mq^*}+\frac{\alpha^*}{mn} }&\lesssim \left( \frac{|B_k|}{|B_{k+j}|}\right)^{\left(\frac{1}{mq^*}+\frac{\alpha^*}{mn}  \right).\frac{\delta_1-1}{\delta_1}}=2^{-jn\left(\frac{1}{mq^*}+\frac{\alpha^*}{mn}  \right).\frac{\delta_1-1}{\delta_1}}\\
\left( \frac{\omega(B_{k+j})}{\omega(B_{k})} \right)^{\frac{\alpha
^*}{mn}}&\lesssim \left( \frac{|B_{k+j}|}{|B_{k}|}\right)^{\frac{\alpha
^*}{mn}.\frac{\delta_1-1}{\delta_1}}=2^{j\frac{\alpha^*}{m}.\frac{\delta_1-1}{\delta_1}}\\
\left(\frac{v(B_k)}{v(B_{k+j})} \right)^{\frac{\alpha^*}{mn} }&\lesssim  \left( \frac{|B_{k}|}{|B_{k+j}|}\right)^{\frac{\alpha^*}{mn}  .\eta}=2^{-j\eta\frac{\alpha^*}{m}}.
\end{align*}
For $j\le 0$, we get
\begin{align*}
\left(\frac{\omega(B_k)}{\omega(B_{k+j})} \right)^{\frac{1}{mq^*}+\frac{\alpha^*}{mn} }&\lesssim \left( \frac{|B_k|}{|B_{k+j}|}\right)^{\left(\frac{1}{mq^*}+\frac{\alpha^*}{mn}  \right)\xi}=2^{-j\xi n \left(\frac{1}{mq^*}+\frac{\alpha^*}{mn}  \right)}\\
\left( \frac{\omega(B_{k+j})}{\omega(B_{k})} \right)^{\frac{\alpha
^*}{mn}}&\lesssim \left( \frac{|B_{k+j}|}{|B_{k}|}\right)^{\frac{\alpha
^*}{mn}.\xi}=2^{j\xi\frac{\alpha^*}{m}}\\
\left(\frac{v(B_k)}{v(B_{k+j})} \right)^{\frac{\alpha^*}{mn} }&\lesssim  \left( \frac{|B_{k}|}{|B_{k+j}|}\right)^{\frac{\alpha^*}{mn}  .\frac{\delta_2-1}{\delta_2}}=2^{-j\frac{\alpha^*}{m}.\frac{\delta_2-1}{\delta_2}}.
\end{align*}
Consequently,
\begin{align*}
\mathcal{D}_k^j \lesssim \begin{cases} 2^{\frac{1}{m}\left(-\frac{j\xi n}{q^*}-j\alpha^*.\frac{\delta_2-1}{\delta_2}\right)}, &j\le 0\\
2^{\frac{1}{m}\left(-\frac{jn}{q^*}.\frac{\delta_1-1}{\delta_1}-j\eta\alpha^*\right)},&j\ge 1. \end{cases}
\end{align*}
Now, let us estimate $\mathcal{E}_{1i}$ for  the case 1. Then, we have
\begin{align*}
\mathcal{E}_{1i}
&\lesssim  \int_{\{y:\|A_i(y)\|<1\}}\frac{|\Phi(y)|}{|y|^n}   |\det A_i^{-1}(y)|^{\frac{m\xi}{q_i}}\|A_i(y) \| ^{\frac{m\xi n}{q_i}}\times\\ 
 &\;\;\;\;\;\;\;\;\;\;\;\times\left( \sup_{k_0\in\mathbb Z}v(B_{k_0})^{-\frac{\lambda_i}{n}}\sum\limits_{k=-\infty}^{k_0} 
 \left(2^{\frac{1}{m}\left(-\frac{jn}{q^*}\frac{\delta_1-1}{\delta_1}-j\alpha^*\frac{\delta_1-1}{\delta_1}-j\alpha^*\frac{\delta_2-1}{\delta_2}+j\xi\alpha^*\right)}\right.\right.\times\\
&\;\;\;\;\;\;\;\;\;\;\;\;\;\;\;\;\;\;\;\;\;\;\;\;\;\;\;\;\;\;\;\;\;\;\;\;\;\;\;\;\;\;\;\;\left.\left.\times\sum\limits_{\ell=-\infty}^j 2^{(j-\ell)\alpha_i\frac{\delta_2-1}{\delta_2}}  v(B_{k+\ell})^{\frac{\alpha_i}{n} }\|f_i\|_{L_\omega^{q_i}(C_{k+\ell})}          \right)^{p_i}  \right)^{\frac{m}{p_i}} dy      \\
&\;\;\;\;\;+\int_{\{y:\|A_i(y)\|\geq 1\}}\frac{|\Phi(y)|}{|y|^n}   |\det A_i^{-1}(y)|^{\frac{m\xi}{q_i}}\|A_i(y) \| ^{\frac{m\xi n}{q_i}}\times\\ 
&\;\;\;\;\;\;\;\;\;\;\times \left(\sup_{k_0\in\mathbb Z}v(B_{k_0})^{-\frac{\lambda_i}{n}} \sum\limits_{k=-\infty}^{k_0} \left(2^{\frac{1}{m}\left(-j\xi\alpha^*-j\eta\alpha^*-\frac{j\xi n}{q^*}+j\alpha^*\frac{\delta_1-1}{\delta_1}\right)}\right.\right.\times\\
&\;\;\;\;\;\;\;\;\;\;\;\;\;\;\;\;\;\;\;\;\;\;\;\;\;\;\;\;\;\;\;\;\;\;\;\;\;\;\;\;\;\;\;\;\;\;\;\;\;\left.\left.\times\sum\limits_{\ell=-\infty}^j 2^{(j-\ell)\alpha_i\frac{\delta_2-1}{\delta_2}}  v(B_{k+\ell})^{\frac{\alpha_i}{n} }\|f_i\|_{L_\omega^{q_i}(C_{k+\ell})}          \right)^{p_i}  \right)^{\frac{m}{p_i}} dy      \\
&\lesssim  \int_{\{y:\|A_i(y)\|<1\}}\frac{|\Phi(y)|}{|y|^n}   |\det A_i^{-1}(y)|^{\frac{m\xi}{q_i}}\|A_i(y) \| ^{\frac{m\xi n}{q_i}}\|A_i(y) \|^{-\left(\left(\frac{n}{q^*}+\alpha^*\right)\frac{\delta_1-1}{\delta_1}+\alpha^*\frac{\delta_2-1}{\delta_2}-\xi\alpha^*\right)}\times\\
&\;\;\;\;\;\;\;\;\;\times
 \left(\sup_{k_0\in\mathbb Z}v(B_{k_0})^{-\frac{\lambda_i}{n}}\sum\limits_{k=-\infty}^{k_0}            \left(
\sum\limits_{\ell=-\infty}^j 2^{(j-\ell)\alpha_i\frac{\delta_2-1}{\delta_2}}  v(B_{k+\ell})^{\frac{\alpha_i}{n} }\|f_i\|_{L_\omega^{q_i}(C_{k+\ell})}            \right)^{p_i} \right)^{\frac{m}{p_i}}dy+\\
&\;\;\;\;\; + \int_{\{y:\|A_i(y)\|\geq 1\}}\frac{|\Phi(y)|}{|y|^n}   |\det A_i^{-1}(y)|^{\frac{m\xi}{q_i}}\|A_i(y) \| ^{\frac{m\xi n}{q_i}}\|A_i(y) \|^{-\left(\left(\frac{n}{q^*}+\alpha^*\right)\xi+\eta\alpha^*-\alpha^*\frac{\delta_1-1}{\delta_1}\right)}\times\\
&\;\;\;\;\;\;\;\;\;\times
 \left(\sup_{k_0\in\mathbb Z}v(B_{k_0})^{-\frac{\lambda_i}{n}}\sum\limits_{k=-\infty}^{k_0}           \left(
\sum\limits_{\ell=-\infty}^j 2^{(j-\ell)\alpha_i\frac{\delta_2-1}{\delta_2}}  v(B_{k+\ell})^{\frac{\alpha_i}{n} }\|f_i\|_{L_\omega^{q_i}(C_{k+\ell})}
\right)^{p_i} \right)^{\frac{m}{p_i}}dy.
\end{align*}
Notice that $\|A_i(y)\|^{-1}\simeq 2^{-j}$ and $
\sum\limits_{\ell=-\infty}^j 2^{(j-\ell)\alpha_i.\frac{\delta_2-1}{\delta_2}}=\frac{1}{1-2^{\alpha_i.\frac{\delta_2-1}{\delta_2}}},$ for all $\alpha_i<0$, and 
$$v(B_{k_0})^{-\frac{\lambda_i}{n}}=\left(\frac{v(B_{k_0+\ell})}{v(B_{k_0})}\right)^{\frac{\lambda_i}{n}}v(B_{k_0+\ell})^{-\frac{\lambda_i}{n}}\leq \left(\frac{v(B_{k_0+j})}{v(B_{k_0})}\right)^{\frac{\lambda_i}{n}}v(B_{k_0+\ell})^{-\frac{\lambda_i}{n}},$$
for all $\ell\leq j$. It is easy to show that
\begin{align*}
\left(\frac{v(B_{k_0+j})}{v(B_{k_0})}\right)^{\frac{\lambda_i}{n}}\lesssim \begin{cases} 2^{j\lambda_i\frac{\delta_2-1}{\delta_2}}, &j\le 0\\
2^{j\eta\lambda_i},\qquad &j\ge 1. \end{cases}
\end{align*}
Then, by the Minkowski inequality for $p_i\geq 1$, it follows that
\begin{align*}
\mathcal{E}_{1i} 
 &\lesssim  \int_{\{y:\|A_i(y)\|<1\}}\frac{|\Phi(y)|}{|y|^n} |\det A_i^{-1}(y)|^{\frac{m\xi}{q_i}}\|A_i(y) \| ^{\frac{m\xi n}{q_i}}\|A_i(y) \|^{-\left(\left(\frac{n}{q^*}+\alpha^*\right)\frac{\delta_1-1}{\delta_1}+(\alpha^*-m\lambda_i)\frac{\delta_2-1}{\delta_2}-\xi\alpha^*\right)}\times\\
&\;\;\;\;\times \left(\sup_{k_0\in\mathbb Z}
\sum\limits_{\ell=-\infty}^j 2^{(j-\ell)\alpha_i\frac{\delta_2-1}{\delta_2}}  \left(v(B_{k_0+\ell})^{-\frac{\lambda_i}{n} }\sum\limits_{k=-\infty}^{k_0}          
 v(B_{k+\ell})^{\frac{\alpha_i p_i}{n} }\|f_i\|^{p_i}_{L_\omega^{q_i}(C_{k+\ell})}            \right)^{\frac{1}{p_i}} \right)^{m}dy+\\
& + \int_{\{y:\|A_i(y)\|\geq 1\}}\frac{|\Phi(y)|}{|y|^n}   |\det A_i^{-1}(y)|^{\frac{m\xi}{q_i}}\|A_i(y) \| ^{\frac{m\xi n}{q_i}}\|A_i(y) \|^{-\left(\xi(\frac{n}{q^*}+\alpha^*)+\eta(\alpha^*-m\lambda_i)-\alpha^*\frac{\delta_1-1}{\delta_1}\right)}\times\\
&\;\;\;\;\;\;\;\;\;\;\;\;\;\;\;\times \left(\sup_{k_0\in\mathbb Z}\sum\limits_{\ell=-\infty}^j 2^{(j-\ell)\alpha_i\frac{\delta_2-1}{\delta_2}} v(B_{k_0+\ell})^{-\frac{\lambda_i}{n}} \left(\sum\limits_{k=-\infty}^{k_0}           
 v(B_{k+\ell})^{\frac{\alpha_i p_i}{n} }\|f_i\|^{p_i}_{L_\omega^{q_i}(C_{k+\ell})}            \right)^{\frac{1}{p_i}} \right)^{m}dy\\
&\lesssim 
\left(\int_{\{y:\|A_i(y)\|<1\}}\frac{|\Phi(y)|}{|y|^n}   |\det A_i^{-1}(y)|^{\frac{m\xi}{q_i}}\|A_i(y) \| ^{\frac{m\xi n}{q_i}}\|A_i(y) \|^{-\left(\left(\frac{n}{q^*}+\alpha^*\right)\frac{\delta_1-1}{\delta_1}+(\alpha^*-m\lambda_i)\frac{\delta_2-1}{\delta_2}-\xi\alpha^*\right)}dy\right. \\
&\;\;\;\;\; \left.+\int_{\{y:\|A_i(y)\|\geq 1\}}\frac{|\Phi(y)|}{|y|^n}   |\det A_i^{-1}(y)|^{\frac{m\xi}{q_i}}\|A_i(y) \| ^{\frac{m\xi n}{q_i}}\|A_i(y) \|^{-\left(\xi(\frac{n}{q^*}+\alpha^*)+\eta(\alpha^*-m\lambda_i)-\alpha^*\frac{\delta_1-1}{\delta_1}\right)}dy\right)\times\\
&\;\;\;\;\;\;\;\;\;\;\;\;\;\;\;\;\;\;\;\; \;\;\;\;\;\;\;\;\;\;  \times \|f_i\|^m_{M\dot{K}_{v,\omega}^{\alpha_i, \lambda_i, p_i,q_i}(\mathbb{R}^n)}.
\end{align*}
Finally, we obtain
\begin{align*}
\|{\mathcal{H}}_{\Phi,\vec{A}}(\vec{f})\|_{M\dot{K}_{v,\omega}^{\alpha^*,\lambda^*,  p,q^*}(\mathbb{R}^n)}&\lesssim \prod_{i=1}^m \left(\int_{\mathbb{R}^n}\frac{|\Phi(y)|}{|y|^n}  |\det A_i^{-1}(y)|^{\frac{m\xi}{q_i}}\|A_i(y) \| ^{\frac{m\xi n}{q_i}}\mathcal{B}_{1i}(y)dy\right)^{\frac{1}{m}}\\
 &\;\;\;\;\;\;\;\;\;\times\prod_{i=1}^m \|f_i\|_{M\dot{K}_{v,\omega}^{\alpha_i,\lambda_i, p_i,q_i}(\mathbb{R}^n)}\\
 &\lesssim\mathcal{C}_{5.1}\prod_{i=1}^m \|f_i\|_{M\dot{K}_{v,\omega}^{\alpha_i,\lambda_i, p_i,q_i}(\mathbb{R}^n)}.
\end{align*}
By the same arguments as case 1, we also obtain the case 2 with the condition  $\frac{\alpha^*}{n}+\frac{1}{q^*} >0$. More precisely, the following is true.
\begin{align*}
\|{\mathcal{H}}_{\Phi,\vec{A}}(\vec{f})\|_{M\dot{K}_{v,\omega}^{\alpha^*,\lambda^*,p,q^*}(\mathbb{R}^n)}&\lesssim \prod_{i=1}^m \left(\int_{\mathbb{R}^n}\frac{|\Phi(y)|}{|y|^n} |\det A_i^{-1}(y)|^{\frac{m\xi}{q_i}}\|A_i(y) \| ^{\frac{m\xi n}{q_i}}\mathcal{B}_{2i}(y)dy\right)^{\frac{1}{m}} \\
 &\;\;\;\;\;\;\;\;\;\;\;\;\;\;\times\prod_{i=1}^m \|f_i\|_{M\dot{K}_{v,\omega}^{\alpha_i,\lambda_i,p_i,q_i}(\mathbb{R}^n)}\\
  &\lesssim\mathcal{C}_{5.2}\prod_{i=1}^m \|f_i\|_{M\dot{K}_{v,\omega}^{\alpha_i,\lambda_i, p_i,q_i}(\mathbb{R}^n)}.
\end{align*}
Therefore, the proof of the theorem is completed.
\end{proof}

In the special case $v=\omega$, we also obtain some sufficient conditions for the boundedness of ${\mathcal{H}}_{\Phi,\vec{A}}$ on the weighted Morrey-Herz spaces
 which are actually better than ones of Theorem \ref{Morrey-Herz-Muckenhoupt}. More precisely, we have the following result.
\begin{theorem}\label{Morrey-Herz-Muckenhoupt1}
Let $1\le q^*, \xi<\infty, \alpha_i<0$, $\lambda_i\geq 0$, for all $i=1,...,m$, and $\omega\in A_\xi$ with the finite reverse H\"{o}lder critical index $r_\omega$.  Assume that $q>q^*\xi r'_\omega, \delta\in (1,r_\omega)$ and $\alpha^*, \lambda^* $ are two real numbers satisfying
$$\lambda^*=\lambda_1+\cdots+\lambda_m \textit{ and  } \;\frac{\alpha^*}{n} +\frac{1}{q^*}=\frac{\sum_{i=1}^m\alpha_i}{n} +\frac{1}{q}.$$
If $\frac{\alpha_i}{n}+ \frac{1}{q_i}\le 0$, for all $i=1,...,m$, and
\begin{align*}
\mathcal{C}_{6.1} = \prod_{i=1}^m \left(\int_{\mathbb{R}^n}\frac{|\Phi(y)|}{|y|^n}   |\det A_i^{-1}(y)|^{\frac{m\xi}{q_i}}\|A_i(y) \| ^{\frac{m\xi n}{q_i}}.\Psi_{1i}(y)dy\right)^{\frac{1}{m}}<\infty,
\end{align*}
where
\begin{align*}
\Psi_{1i}(y)&=\|A_i(y) \|^{m\left(\lambda_i-n\left(\frac{\alpha_i}{n}+\frac{1}{q_i}\right)\left(\frac{\delta-1}{\delta}\right)\right)}\chi_{ \{y\in \mathbb{R}^n:\|A_i(y)\|<1\}}\;+\\
 &\;\;\;+
\|A_i(y) \|^{m\xi\left(\lambda_i-n\left(\frac{\alpha_i}{n}+\frac{1}{q_i}\right)\right)}\chi_{ \{y\in \mathbb{R}^n:\|A_i(y)\|\ge 1\}},
\end{align*}
or $\frac{\alpha_i}{n}+ \frac{1}{q_i} > 0$, for all $i=1,...,m$, and 
\begin{align*}
\mathcal{C}_{5.2} =  \prod_{i=1}^m \left(\int_{\mathbb{R}^n}\frac{|\Phi(y)|}{|y|^n}   |\det A_i^{-1}(y)|^{\frac{m\xi}{q_i}}\|A_i(y) \| ^{\frac{m\xi n}{q_i}}.\Psi_{2i}(y)dy\right)^{\frac{1}{m}}<\infty,
\end{align*}
where
\begin{align*}
\Psi_{2i}(y)&=\|A_i(y) \|^{m\left(\lambda_i\left(\frac{\delta-1}{\delta}\right)-n\xi\left(\frac{\alpha_i}{n}+ \frac{1}{q_i}\right)\right)}\chi_{ \{y\in \mathbb{R}^n:\|A_i(y)\|<1\}}\;+ \\
&\;\;+\|A_i(y) \|^{m\left(\lambda_i\xi-n\left(\frac{\alpha_i}{n}+ \frac{1}{q_i}\right)\left(\frac{\delta-1}{\delta}\right)\right)}\chi_{ \{y\in \mathbb{R}^n:\|A_i(y)\|\ge 1\}},
\end{align*}
then we have ${\mathcal{H}}_{\Phi,\vec{A}}$ is bounded from $\prod_{i=1}^m  M\dot{K}_{p_i,q_i}^{\alpha_i, \lambda_i}(\omega,\mathbb{R}^n)$ to $M\dot{K}_{p,q^*}^{\alpha^*,\lambda^*}(\omega,\mathbb{R}^n)$.
\end{theorem}
\begin{proof}
Observe that from the condition 
$$\frac{\alpha^*}{n} +\frac{1}{q^*}=\frac{\sum_{i=1}^m\alpha_i}{n} +\frac{1}{q}=\frac{\sum_{i=1}^m\alpha_i}{n}+\sum_{i=1}^m\frac{1}{q_i},$$
we have 
$$\omega(B_k)^{\left(\frac{\alpha^*}{n}+\frac{1}{q^*}\right)}=\prod_{i=1}^m\omega(B_k)^{\left(\frac{\alpha_i}{n}+\frac{1}{q_i}\right)}.$$
The other arguments are proved in the same way as Theorem \ref{Morrey-Herz-Muckenhoupt}. Thus, its proof is omitted here and left the details to the interested reader.
\end{proof}

As a consequence, by letting $\lambda_1=\cdots=\lambda_m=0$, we also obtain some sufficient conditions the same as Theorem \ref{Morrey-Herz-Muckenhoupt} and Theorem \ref{Morrey-Herz-Muckenhoupt1} for the boundedness of the  multilinear Hausdorff operators on the two weighted Herz spaces with the Muckenhoupt weights.
\vskip 10pt
{\textbf{Acknowledgments}}. The first author of this paper is supported by the Vietnam National Foundation for Science and Technology Development (NAFOSTED) under Grant No 101.02-2014.51.

\bibliographystyle{amsplain}

\end{document}